\documentclass{amsart}
\usepackage{amsmath,amsfonts,amssymb}
\usepackage{latexsym}
\usepackage{epsfig,overpic}
\usepackage{stmaryrd}
\usepackage{color}
\usepackage{graphicx}

\numberwithin{equation}{section}

  \newcounter{mnote}
  \let\oldmarginpar\marginpar
    \renewcommand\marginpar[1]{\-\oldmarginpar[\raggedleft\footnotesize #1]%
    {\raggedright\footnotesize #1}}

\def\XXint#1#2#3{{\setbox0=\hbox{$#1{#2#3}{\int}$ }
\vcenter{\hbox{$#2#3$ }}\kern-.6\wd0}}

\newcommand{\R}{\mathcal R}
\newcommand{\E}{\mathcal E}
\newcommand{\T}{\mathcal T}
\newcommand{\A}{\mathcal A}

\newcommand{\Div}{\mathop{\rm div}}
\newcommand{\curl}{\mathop{\rm curl}}

\newcommand{\rot}{\mathop{\rm rot}}

\definecolor{darkred}{rgb}{.7,0,0}

\definecolor{green}{rgb}{0,0.7,0}

\def\B{\mathfrak{B}}
\def\H{\mathfrak{H}}
\def\Z{\mathfrak{Z}}

\def\M{\mathcal{M}}

\def\trace{\hspace{1pt}{\rm tr}\hspace{2pt}}

\newcommand{\la}{\langle}
\newcommand{\ra}{\rangle}

\newtheorem{theorem}{Theorem}
\newtheorem{lemma}[theorem]{Lemma}
\newtheorem{corollary}[theorem]{Corollary}
\newtheorem{proposition}[theorem]{Proposition}
\newtheorem{remark}{Remark}

\title[AFEM for harmonic forms]{Convergence and quasi-optimality of adaptive finite element methods for harmonic forms}

\author[A.~Demlow]{Alan Demlow$^1$}

\thanks{$^1$Partially supported by National Science Foundation grant DMS-1518925.}

\address{
  Department of Mathematics, Texas A\&M University, College Station, TX  77843
  }

\email{demlow@math.tamu.edu}



\begin{document}

\begin{abstract}  Numerical computation of harmonic forms (typically called harmonic fields in three space dimensions) arises in various areas, including computer graphics and computational electromagnetics.  The finite element exterior calculus framework also relies extensively on accurate computation of harmonic forms.  In this work we study the convergence properties of adaptive finite element methods (AFEM) for computing harmonic forms.  We show that a properly defined AFEM is contractive and achieves optimal convergence rate beginning from any initial conforming mesh.  This result is contrasted with related AFEM convergence results for elliptic eigenvalue problems, where the initial mesh must be sufficiently fine in order for AFEM to achieve any provable convergence rate.  
 \end{abstract}

\maketitle


\section{Introduction}

In this paper we prove convergence and rate-optimality of adaptive finite element methods (AFEM) for computing harmonic forms.  Let $\Omega \subset \mathbb{R}^3$ be a polyhedral domain with Lipschitz boundary $\partial \Omega$.  The space $\H^1$ of harmonic forms (the first de Rham cohomology group in $\mathbb{R}^3$) is
\begin{align}
\label{eq0-1}
\H^1 = \{ {\bf v} \in L_2(\Omega)^3 | \curl {\bf v} = {\bf 0}, ~ \Div ({\bf v}) = 0, ~ {\bf v} \cdot {\bf n} = 0  \hbox{ on } \partial \Omega \}.
\end{align}
Here ${\bf n}$ is the outward unit normal on $\partial \Omega$.  
$\beta^1 :={\rm dim}\hspace{2pt}\H^1<\infty$ is equal to the number of handles of the domain $\Omega$, so $\beta^1=0$ if $\Omega$ is simply connected.   We denote by $\{q^1, ..., q^{\beta^1}\}$ a fixed $L_2$-orthonormal basis for $\H^1$.  

Computation of harmonic forms arises in applications including computer graphics and surface processing \cite{XZCX09, FSDH07} and numerical solution of PDE related to problems having Hodge-Laplace structure posed on domains with nontrivial topology.  Important practical examples of the latter type arise in boundary and finite element methods for electromagnetic problems.  There computation of harmonic fields may arise as a necessary precursor step which in essence factors nontrivial topology out of subsequent computations.  We refer to \cite{Hip02, HO02, RV10, RBGV13, TV14, BS15} for general discussion, topologically motivated algorithms for efficient computation of $\H^1$, and specific applications requiring such a step.  In addition, harmonic fields on spherical subdomains play an important role in understanding singularity structure of solutions to Maxwell's equations on polyhedral domains \cite[Section 6.4]{CD00}.  Computation of harmonic forms is also an important part of the  finite element exterior calculus (FEEC) framework.  The FEEC framework provides a systematic exposition of the tools needed to stably solve Hodge-Laplace problems related to the de Rham complex and other differential complexes \cite{AFW06, AFW10}.  It thus has provided a broader exposition of many of the tools needed for the numerical analysis of Maxwell's equations.  The mixed methods used in these works solve the Hodge-Laplace problem modulo harmonic forms, so as above their accurate computation is a prerequisite for accurate computation of Hodge-Laplace solutions.  This is true of adaptive computation of solutions to Hodge-Laplace problems also \cite{DH14, MHS13}.  Our results thus serve as a necessary precursor to study of convergence of AFEM for solving the full Hodge-Laplace problem as posed within the FEEC framework.     While keeping in mind the concrete representation \eqref{eq0-1}, outside of the introduction we will use the more general notation and tools that have been developed within the FEEC framework.  The results described in the introduction below are valid essentially verbatim in arbitrary space dimension $n$ and for arbitrary cohomology group $\H^k$, $1 \le k \le n-1$.  

We briefly describe our setting.  Let $\T_\ell$, $\ell \ge 0$, be a set of nested, adaptively generated simplicial decompositions of $\Omega$.  Let also $V_\ell^0 \subset H^1(\Omega)$ and $V_\ell^1 \subset H(\curl;\Omega)$ be conforming finite element spaces which together satisfy a complex property described in more detail below.  For \eqref{eq0-1} we may take $V_\ell^0$ to be $H^1$-conforming piecewise linear functions and $V_\ell^1$ to be lowest-order N\'ed\'elec edge elements.  Higher-degree analogues also may be used.  The space $\H_\ell^1$ of discrete harmonic fields corresponding to $\H^1$ is then those fields $q_\ell \in V_\ell^1$ satisfying
\begin{align}
\label{eq0-2}
\begin{aligned}
\langle q_\ell,  \nabla \tau_\ell \rangle & = 0, ~\tau_\ell \in V_\ell^{0}, 
\\ \langle \curl q_\ell, \curl v_\ell \rangle &  = 0, ~ v_\ell \in V_\ell^1.
\end{aligned}
\end{align}
Note that while $\curl q_\ell ={\bf 0}$, $q_\ell$ is not generally in $H(\Div; \Omega)$.  The first equation in \eqref{eq0-2} instead implies that $\Div_h q_\ell=0$, where $\Div_h$ is a weakly defined discrete divergence operator which is not generally the same as the restriction of $\Div$ to $V_\ell^1$.  We denote by $\{q_\ell^j\}_{j=1}^{\beta^1}$ a (computed) orthonormal basis for $\H_\ell^1$.  Let also $P_{\H_\ell^1}$ be the $L_2$ projection onto $\H_\ell^1$.  Given $q \in \H^1$, a posteriori error estimates for controlling $\|q-P_{\H_\ell^1} q\|_{L_2(\Omega)}$ were given in \cite{DH14}.  From these we may generate an adaptive finite element method having the standard form $\textsf{solve} \rightarrow \textsf{estimate} \rightarrow \textsf{mark} \rightarrow \textsf{refine}$ for controlling the defect between $\H^1$ and $\H_\ell^1$.  

There are important technical and conceptual parallels between AFEM for controlling harmonic forms and AFEM for elliptic eigenvalue problems, as both consider approximation of finite-dimensional (invariant) subspaces.  Consider the eigenvalue problem $-\Delta u = \lambda u$ with Dirichlet boundary conditions.  Assume that $\lambda$ is the $j$-th eigenvalue of the continuous problem.  AFEM employing residual-type error indicators based on the $j$-th discrete eigenvalue $\lambda_\ell$ on the $\ell$-th mesh level yield an approximating sequence $(\lambda_\ell, u_\ell)$ for the $j$-th continuous eigenpair $(\lambda, u)$.  Analyses of AFEM for elliptic eigenvalue problems have focused on two convergence regimes.    The first is a preasymptotic regime in which the method converges, but with no provable rate.  In \cite{GMZ09} it was proved that starting from {\it any} initial mesh $\T_0$, $(\lambda_\ell, u_\ell) \rightarrow (\lambda, u)$ in $\mathbb{R} \times H_0^1(\Omega)$, but with no guaranteed convergence rate.  In the second convergence regime the eigenvalue problem is effectively linearized, and the convergence behavior is that of a source problem.  More precisely, it the initial mesh $\T_0$ is sufficiently fine,  then the AFEM is contractive and achieves an optimal convergence rate \cite{DXZ08, GG09,  DHZ15, Gal15, BD15}.  The plain convergence results of \cite{GMZ09} guarantee that such a state is initially reached from any initial mesh $\T_0$.  

Our results below indicate that the convergence behavior of harmonic forms essentially begins in a transition region in which the AFEM contracts, but with contraction constant that may improve as the overlap between $\H^1$ and $\H_\ell^1$ increases.  A regime in which AFEM contracts with constants independent of essential quantities is then eventually reached, as for AFEM for eigenvalue problems.  It is to be expected that AFEM for harmonic forms contract from any initial mesh as computation of harmonic forms reduces to solving $Au=0$ with $A$ a linear operator.  That is, harmonic forms are eigenfunctions with {\it known} eigenvalue and their computation is essentially  a linear problem.    On the other hand, we have set as our main goal the production of an {\it orthonormal} basis for $\H^1$ and do not assume any particular alignment or method of production for the basis.  The problem of producing an orthonormal basis is mildly nonlinear, so it is reasonable that the contraction constant can improve as the mesh is refined.  As we discuss in \S\ref{subsec:cutting}, alternate methods for producing and aligning the discrete and continuous bases may lead to AFEM with different properties.  Our framework has the advantage of being completely generic with respect to the method used to compute $\H_\ell^1$.

There are two main challenges in proving contraction of AFEM for eigenspaces.  These are lack of orthogonality caused by non-nestedness of the discrete eigenspaces, and lack of alignment of computed eigenbases at adjacent discrete levels.  In the context of elliptic eigenvalue problems, suitable nestedness and orthogonality is recovered only on sufficiently fine meshes as the nonlinearity of the problem is resolved.  In the case of harmonic forms a novel situation arises.  The discrete spaces $\{H_\ell^1\}$ of harmonic forms are not themselves nested.   However, as we show below the Hodge decomposition nonetheless guarantees sufficient nestedness and orthogonality uniformly starting from any initial mesh.  In essence, {\it topological} resolution of $\Omega$ by $\T_0$ is sufficient to ensure some {\it analytical} resolution of $\H^1$ by $\H_0^1$. 

Lack of alignment between discrete bases at adjacent mesh levels occurs in the case of multi-dimensional target subspaces, including multiple or clustered eigenvalues and harmonic forms when $\beta>1$.  Standard AFEM convergence proofs employ ``indicator continuity'' arguments which in our case would require comparing discrete basis members $q_\ell^j$ and $q_{\ell+1}^j$ on adjacent mesh levels.  When $\beta^1>1$, $q_\ell^j$ and $q_\ell^{j+1}$ may not be meaningfully related.  To overcome the difficulty of multiple eigenvalues, we follow \cite{DHZ15, Gal15, BD15} in using a non-computable error estimator $\mu_\ell$ calculated with respect to projections $P_\ell q^j$ of a {\it fixed} basis for the continuous harmonic forms.  Indicator continuity arguments apply to these theoretical error indicators, which must in turn be compared to the practical ones based on $\{q_\ell^j\}$.  We follow \cite{BD15} in establishing an equivalence between the theoretical and practical indicators with constants asymptotically independent of essential quantities.

We now briefly describe our results.  We first show that there exists $\gamma>0$ and $0< \rho <1$ such that 
\begin{align}
\label{intro:contraction}
\sum_{j=1}^{\beta^1} \|q^j-P_{\ell+1} q^j\|^2 + \gamma \mu_{\ell+1}^2 \le \rho \left (\sum_{j=1}^{\beta^1} \|q^j - P_{i} q^j\|^2 + \gamma \mu_\ell^2 \right ).
\end{align}
While we may take $\gamma, \rho$ above independent of mesh level, our proof indicates that $\rho$ may in fact improve (decrease) as the overlap between $\H^1$ and $\H_\ell^1$ improves.  Because a contraction occurs from the initial mesh with fixed constant $\rho<1$, this improvement in overlap is guaranteed to occur.  Thus our result lies between standard results for elliptic source problems in which contraction occurs from the initial mesh with fixed contraction constant, and AFEM convergence results for elliptic eigenvalue problems for which sufficient overlap between discrete and continuous eigenspaces is not guaranteed unless the initial mesh is sufficiently fine.  

We additionally prove rate-optimality, or more precisely that
\begin{align}
\label{intro_optimality}
 \left ( \sum_{j=1}^{\beta^1} \|q^j-P_{\ell+1} q^j\|^2 \right )^{1/2} \le C (\# \T_\ell -\# \T_0)^{-s}
\end{align}
whenever systematic bisection is able to produce a sequence of meshes that similarly approximates $\H^1$ with rate $s$.  As is typical for AFEM results, our proof of \eqref{intro_optimality} requires that the D\"orfler (bulk) marking parameter that specifies the fraction of elements to be refined in each step of the algorithm be sufficiently small.  We show that the threshold value for the D\"orfler parameter is independent of all essential quantities, including the initial resolution of $\H^1$ by $\H_\ell^1$.  This situation is typical of elliptic source problems; cf. \cite{CKNS08}.  In contrast to elliptic source problems, however, the constant $C$ may depend on the quality of approximation of $\H^1$ by $\H_\ell^1$.  Finally, proof of rate optimality requires establishing a localized a posteriori upper bound for the defect between the target spaces on nested meshes.  This step is substantially more involved for general problems of Hodge-Laplace type than for standard scalar problems.  Proof of localized upper bounds has been previously given for Maxwell's equations \cite{ZCSWX11}, but the necessary tools have not previously appeared in a form suitable for our purposes.  Our approach below is valid for arbitrary form degree and space dimension and generically for problems of Hodge-Laplace type.  In addition to being more general, our proof is also slightly simpler due to its use of the recently-defined quasi-interpolant of Falk and Winther \cite{FW14}.

The remainder of the paper is organized as follows.  In \S\ref{sec:deRham} we describe continuous and discrete spaces of differential forms, interpolants into discrete spaces of differential forms, and existing a posteriori estimates for harmonic forms.  In \S\ref{sec:L2} we give a number of preliminary results concerning the Hodge decomposition and $L_2$ projections described above.  Section \S\ref{sec:contraction} and \S\ref{sec:optimality} contain statements, discussion, and proofs of our contraction and optimality results.  Section \S\ref{sec:extensions} contains brief discussion of essential boundary conditions and harmonic forms in the presence of coefficients, and also of the effects of methods for computing harmonic fields on the resulting AFEM.  Finally, in \S\ref{sec:numerics} we present numerical results illustrating our theory.

\section{The de Rham complex and its finite element approximation}
\label{sec:deRham}

In this section we generalize the classical function space setting from the introduction to arbitrary space dimension and form degree and introduce corresponding finite element spaces and tools.  We for the most part follow \cite{AFW10} in our notation and refer to that work for more detail.

\subsection{The de Rham complex}

Let $\Omega$ be a bounded Lipschitz polyhedral domain in $\mathbb{R}^n$, $n \ge 2$.   Let $\Lambda^k(\Omega)$ represent the space of smooth $k$-forms on $\Omega$.  The natural $L_2$ inner product is denoted by $\la \cdot, \cdot \ra$, the $L_2$ norm by $\| \cdot \|$, and the corresponding space by $L_2 \Lambda^k(\Omega)$.  We let $d$ be the exterior derivative, and $H \Lambda^k(\Omega)$ be the domain of $d^k$ consisting of $L_2$ forms $\omega$ for which $d \omega \in L_2 \Lambda^{k+1}(\Omega)$.  We denote by $\|\cdot \|_H$ the associated graph norm; here one may concretely think of $H$ as $H(\curl)$, $H(\Div)$, or $H^1$.   
We denote by $W_p^r \Lambda^k(\Omega)$ the corresponding Sobolev spaces of forms and set $H^r \Lambda^k(\Omega) = W_2^r \Lambda^k (\Omega)$.  Finally, for $\omega \subset \mathbb{R}^n$, we let $\|\cdot\|_\omega=\|\cdot \|_{L_2 \Lambda^k(\omega)}$ and $\|\cdot \|_{H, \omega}=\|\cdot \|_{H \Lambda^k(\omega)}$; in both cases we omit $\omega$ when $\omega=\Omega$.    
 
 Denote by $\delta$ the codifferential, that is, the adjoint of the exterior derivative $d$ with respect to $\langle \cdot, \cdot \rangle$.
 The space of harmonic $k$-forms is given by
 \begin{align}
 \label{eq2-1}
 \H^k = \{q \in H \Lambda^k(\Omega): dq=0, ~\delta q=0, ~\trace \star q=0\}.
 \end{align}
Here $\trace$ is the trace operator and $\star$ the Hodge star operator.  We denote by 
\begin{align}
\beta^k = {\rm dim} \hspace{2pt} \H^k
\end{align}
the $k$-th Betti number of $\Omega$.  When $n=3$, $\beta^0=1$, $\beta^1$ is the number of holes in $\Omega$, $\beta^2$ the number of voids, and $\beta^3=0$.  In general, $\beta<\infty$.  We additionally let $\B^k = d H\Lambda^{k-1}$ be the range of $d^{k-1}$, $\Z^k = {\rm kernel} d^k$ the kernel of $d^k$, and by $\Z^{k \perp}$ the orthogonal complement of $\Z^k$ in $H\Lambda^k$.  We have $\Z^k= \H^k \oplus \B^k$, and the Hodge decomposition is given by $H \Lambda^k = \B^k \oplus \H^k \oplus \Z^{k \perp}$.  Note that $\B^k \subset \Z^k$, that is, $d \circ d=0$.

\subsection{Meshes and mesh properties}

We employ, and now briefly review, the conforming simplicial mesh refinement framework commonly employed in AFEM convergence theory; cf. \cite{Ste08} for details.  Let $\T_0$ be a conforming, shape regular simplicial decomposition of $\Omega$.  By fixing a local numbering of all vertices of all $T \in \T_0$, all possible descendants $\T$ of $\T_0$ that can be created by newest vertex bisection are uniquely determined.  By newest vertex bisection we mean either the refinement procedure as it was developed in two space dimensions or its generalization to any space dimension.  The simplices in any of those partitions are {\em uniformly shape regular}, dependent only on the shape regularity parameters of $\T_0$ and the dimension $n$.  Generally a descendant of $\T_0$ is non-conforming.  Our AFEM will generate a nested sequence $\T_0 \subset \T_1 \subset \T_2...$ of conforming meshes which we index by $\ell$.  Given marked sets $\M_\ell \subset \T_\ell$, conforming bisection thus refines all $T \in \M_\ell$ and then additional elements in order to ensure that $\T_{\ell+1}$ is conforming.  With a suitable numbering of the vertices in the initial partition, the total number of refinements needed to make the sequence $\{\T_\ell\}$ conforming can be bounded by the number of marked elements.  More precisely, for any $\underline{\ell} \ge 0$, there is a constant $C_{ref,\underline{\ell}}$ such that for $\ell > \underline{\ell}$, 
\begin{align}
\label{numrefined}
\# \T_\ell - \# \T_{\underline{\ell}} \le C_{ref, \underline{\ell}} \sum_{i=\underline{\ell}}^{\ell-1} \# \M_i
\end{align} 
with constant depending on $\T_{\underline{\ell}}$.  Such an initial numbering always exists, possibly after an initial uniform refinement of $\T_0$.  Assuming such a numbering, we denote the set of all {\em conforming} descendants $\T$ of $\T_0$ by $\mathbb{T}$.
For $\T, \tilde{\T} \in \mathbb{T}$, we write $\T \subset \tilde{\T}$  when $\tilde{\T}$ is a refinement of $\T$. Finally, for $T \in \T \in \mathbb{T}$ we let $h_T=|T|^{1/n}$.

\subsection{Approximation of the de Rham complex and interpolants}
\label{sec:interpolants}

Given $\T \in \mathbb{T}$, let $...V_\T^{k-1}  \rightarrow V_\T^k \rightarrow V_\T^{k+1} \rightarrow ...$ be an approximating subcomplex of the de Rham complex with underlying mesh $\T$.  That is, $V_\T^k \subset H\Lambda^k(\Omega)$, $v \in V_\T^k$ is polynomial on each $T \in \T$, and $d V_\T^{k-1} \subset V_\T^k$.  When $\T = \T_\ell$ we use the abbreviation $V_{\T_\ell} = V_\ell$ (where here we also depress the dependence on $k$).  We refer to \cite{AFW10} for descriptions of the relevant spaces.  In three space dimensions, we may think for example of standard Lagrange spaces approximating $H\Lambda^0 =H^1$, N\'edel\'ec spaces approximating $H(\curl)$, and Raviart-Thomas or BDM spaces approximating $H(\Div)$, with appropriate relationships between polynomial degrees enforced to ensure commutation properties.  

In addition to the subcomplex property, we also require the existence of a commuting interpolant (or more abstractly, commuting cochain projection) with certain properties.  Desirable properties include commutativity, boundedness on the function spaces $L_2\Lambda$ and/or $H\Lambda$ that are natural for our setting, local definition of the operator, and projectivity, and a local regular decomposition propertyL.   Several recent papers have achieved constructions possessing some of these properties \cite{Sch01, Sch08, CW08, FW14}, though no one operator has been shown to possess all of them.  We use two different such constructions below along with a standard Scott-Zhang interpolant.  First, the commuting projection operator $\pi_{CW}: L_2\Lambda(\Omega) \rightarrow V_\T$ of Christiansen and Winther \cite{CW08} has the following useful properties:
\begin{align}
\label{CW_interp}
d^k \pi_{CW}^k = \pi_{CW}^{k+1} d^k, ~~\pi_{CW}^2 = \pi_{CW}, ~~\|\pi_{CW} v \|_{L_2(\Omega)} \lesssim \|v\|_{L_2(\Omega)}~ \forall~ \omega \in L_2\Lambda(\Omega).
\end{align}
The Christiansen-Winther interpolant also can be modified to preserve homogeneous essential boundary conditions with no change in its other properties.  

The commuting projection operator $\pi_{FW}$ defined by Falk and Winther in \cite{FW14} has the following properties.  First, it commutes:
\begin{align}
\label{FW_interp1}
d^k \pi_{FW}^k = \pi_{FW}^{k+1} d^k.
\end{align}
Second, given $T \in \T$, there is a patch of elements surrounding $T$ such that
\begin{align}
\label{FW_count}
\#(T \subset \omega_T) \lesssim 1.
\end{align}
and for $v \in H \Lambda(\Omega)$, 
\begin{align}
\label{FW_interp2}
\|\pi_{FW} v\|_{T} \lesssim \|v\|_{\omega_T} + h_T \|d \pi_{FW} v\|_{\omega_T}, ~~ \|\pi_{FW}v \|_{H\Lambda(T)} \lesssim \|v\|_{H\Lambda(\omega_T)}.  
\end{align}
The patch $\omega_T$ is not necessarily the standard patch of elements sharing a vertex with $T$, and its configuration may depend on form degree.  The relationship \eqref{FW_count} is however the essential one for our purposes. Given $v \in H^1(\omega_T)$, there also follows the interpolation estimate 
\begin{align}
\label{FW_interp4}
h_T^{-1} \|v-\pi_{FW} v\|_T + |v-\pi_{FW} v|_{H^1 \Lambda(T)}  \lesssim |v|_{H^1 \Lambda(\omega_T)}.  
\end{align}
Finally, $\pi_{FW}$ is a local projection in the sense that
\begin{align}
\label{FW_interp3}
 u|_{\omega_T} \in V_\T|_{\omega_T} ~ \Rightarrow ~(u-\pi_{FW}u)|_T \equiv 0.
 \end{align}
In summary, $\pi_{FW}$ and $\pi_{CW}$ are both commuting projection operators.  $\pi_{FW}$ is however locally defined, but only stable in $H\Lambda$, while $\pi_{CW}$ is globally defined but stable in $L_2 \Lambda$.  Also, to our knowledge there is no detailed discussion in the literature of the modifications needed to construct a version of $\pi_{FW}$ which preserves essential boundary conditions, although a comment in \cite{FW14} indicates that such an adaptation is natural.  

We shall also employ an $L_2$-stable Scott-Zhang operator $I_{SZ}:L_2 \Lambda^k(\Omega) \rightarrow \tilde{V}_\T^k \subset H^1 \Lambda^k(\Omega)$ \cite{SZ90}. Here $\tilde{V}_\T^k$ is the smallest space of forms containing all forms with continuous piecewise linear coefficients.  $I_{SZ}$ may be obtained by employing the standard scalar Scott-Zhang interpolant coefficientwise.  We then have for $u \in L_2\Lambda(\Omega)$
\begin{align}
\label{SZ_interp}
\begin{aligned}
\|I_{SZ} u \|_{L_2\Lambda (T)} & \lesssim \|u\|_{L_2\Lambda (\omega_T)}, 
\\ h_T^{-1} \|u- I_{SZ}u\|_{L_2 \Lambda(T)} + |u-I_{SZ} u|_{H^1\Lambda(T)} & \lesssim  |u|_{H^1\Lambda(\omega_T)}.
\end{aligned}
\end{align}

We shall finally use a regular decomposition property, which is given in \cite[Lemma 5]{DH14} (this lemma is largely a reformulation of more general results given in \cite{MiMiMo08}).  Given a form $v \in H\Lambda^k(\Omega)$, there are $\varphi \in H^1 \Lambda^{k-1}(\Omega)$ and $z \in H^1 \Lambda^k(\Omega)$ such that
\begin{align}
\label{reg_decomp}
v=d\varphi + z, ~~\|\varphi \|_{H^1 \Lambda^{k-1}(\Omega)} + \|z\|_{H^1 \Lambda^k(\Omega)} \lesssim \|v\|_{H\Lambda^k(\Omega)}.
\end{align}

\subsection{Discrete harmonic forms}

 Let $\H_\T^k \subset V_\T^k$ be the set of discrete harmonic forms, which is more precisely defined as those $q_\T \in V_\T^k$ satisfying

\begin{align}
\begin{aligned}
\langle q_\T, d\tau_\T \rangle & = 0, ~\forall \tau_\T \in V_\T^{k-1}, 
\\ \langle d q_\T, d v_\T \rangle &  = 0, ~ \forall v_\T \in V_\T^k.
\end{aligned}
\end{align}
The discrete Hodge decomposition $V_\T^k = \B_\T^k \oplus \H_\T^k \oplus \Z_\T^{k, \perp}$  ($\Z_\T^k = \B_\T^k \oplus \H_\T^k = {\rm kern} \hspace{2pt} d|_{V_\T^k}$) is defined entirely analogously to the continuous version.  Note however that while
\begin{align}
\Z_\T^k \subset \Z^k \hbox{ and } \B_\T^k \subset \B^k,
\end{align}
there holds
\begin{align}
\H_\T^k \not\subset \H^k \hbox{ and } \Z_\T^{k, \perp} \not\subset \Z^{k, \perp}.
\end{align}

We briefly summarize the index system we use. The integer $n$ is the dimension of the domain $\Omega$. The integer $k$, $0\leq k\leq n$, is the order of differential form. The subscript $\ell$, $\ell=0,1,2\ldots$ is used for a sequence of triangulations of the domain $\Omega$ and the corresponding finite element spaces. For a fixed $k$, $0\leq k\leq n$, we consider the convergence of the sequence $\{q_\ell, \ell=0,1,2,\ldots \}$. Therefore we may suppress the index $k$ when no confusion can arise. $j$ is used to index the bases of $\H$ and $\H_\T$. 

\section{Properties of an $L_2$ Projection}
\label{sec:L2}

We first specify the error quantity that we seek to control.  Let $\{q^j\}_{j=1}^{\beta}$ be an orthonormal basis for $\H$, and let $\{q_\T^j\}_{j=1}^{\beta}$ be an orthonormal basis for $\H_\T$.  Denote by $P$ the $L_2$ projection onto $\H$ and by $P_\T = P_{\H_\T}$ the $L_2$ projection onto $\H_\T$.  As above we use the abbreviation $P_\ell = P_{\T_\ell}$ and similarly for $q_\ell^j$.  
We seek to control the subspace defect
\begin{align}
\label{eq2-9}
\E_\T := \left ( \sum_{j=1}^{\beta} \|q^j-P_\T q^j\|^2 \right )^{1/2}  .
\end{align}
\begin{proposition}
\begin{align}
\label{eq2-9-a}
\left ( \sum_{j=1}^{\beta} \|q^j-P_\T q^j\|^2 \right )^{1/2} = \left ( \sum_{j=1}^{\beta} \|q_\T^j-P q_\T^j\|^2 \right )^{1/2}.
\end{align}
\end{proposition}
\begin{proof}
There holds $P_\T q^j = \sum_{m=1}^{\beta} \langle q^j, q_\T^m \rangle q_\T^m$ and $\|P_\T q^j\| ^2= \sum_{m=1}^\beta \langle q^j, q_\T^m \rangle^2$, and similarly $\|P q_\T^m\|^2 = \sum_{j=1}^\beta \langle q_\T^m, q^j \rangle^2$. Also, $\|q^j-P_\T q^j\|^2 = \langle q^j-P_\T q^j, q^j-P_\T q^j \rangle = 1-\|P_\T q^j\|^2$ and similarly for $\|q_\T^m - P q_\T^m\|^2$.  Thus
\begin{align}
\label{eq2-9-c}
\sum_{j=1}^{\beta} \|q^j-P_\T q^j\|^2 = \sum_{j=1}^{\beta} \left (1-\sum_{m=1}^\beta \langle q^j, q_\T^m \rangle^2 \right ) = \sum_{m=1}^{\beta} \|q_\T^m - P q_\T^m\|^2.
\end{align}
\end{proof}

Previous error estimates for errors in harmonic forms have instead controlled the {\it gap} between $\H$ and $H_\ell$.  Given two finite-dimensional subspaces $A, B$ of the same ambient Hilbert space with ${\rm dim} \hspace{2pt} A = {\rm dim} \hspace{2pt} B$, we define
\begin{align}
\delta(A,B)=\sup_{x \in A, \|x\|=1} \|x-P_B x\|, ~~~~{\rm gap}(A,B) = \max\{ \delta(A,B), \delta(B,A)\}.
\end{align} 
There also holds in this case
\begin{align}
\label{gapequiv}
\delta(A,B)=\delta(B,A).  
\end{align}

Let $P_{\Z_\T}$ be the $L_2$ projection onto $\Z_\T$. 
\begin{lemma}
\begin{align}
\label{eq3-4}
P_\T q = P_{\Z_\T} q, \quad \text{for all } q\in \H.
\end{align}
\end{lemma}
\begin{proof}
Because $\Z_\T= \B_\T \oplus \H_\T$, we have for $q \in \H$ that $P_{\Z_\T} q = P_{\B_\T} q + P_{\H_T} q$.  But $\B_\T \subset \B$ and $\H \perp \B$, so $P_{\B_\T} q=0$.  
\end{proof}

%

\begin{proposition} \label{prop3} 
Given $\T \subset \T'$ and $q \in \H$, we have $\| P_{T} q\|\leq \|P_{\T'} q \|$.  In particular, $\|P_\ell q\|\le \|P_{\ell+1} q\|$.    
\end{proposition}
\begin{proof}
For $q \in \H$, $P_\T q = P_{\Z_\T} q$.  Noting that $\Z_\T \subset \Z_{\T"}$ completes the proof.
\end{proof}

We next show that there is overlap between $\H$ and $\H_\ell$ on any conforming mesh $\T_\ell$; cf. \cite[Theorem 2.7]{MHS13}.

\begin{proposition}\label{lem2} 
The projection $P_\T$ is injective. Namely, given $q \in \H$ and $q\neq 0$, $P_\T q \neq 0$.  
\end{proposition}
\begin{proof}
Let $\pi$ be a commuting projection operator (either $\pi_{FW}$ or $\pi_{CW}$ will suffice).  Suppose there exists a $0 \neq q \in \H$ with $P_\ell q = 0$.  By \eqref{gapequiv},  $\delta(\H, \H_\ell) = 1 =\delta(\H_\ell, \H) $, so there is a non-zero form $q_\ell \in \H_\ell$ with $P_{\H}q_\ell = 0$. Note that $\H_\ell \subset \Z_\ell \subset \Z = \B \oplus \H$, so $P_{\H}q_\ell = 0$ implies $q_\ell \in \B$. That is, $q_\ell = d v$ for some $v\in V^{k-1}$. But $q_\ell = \pi^k q_\ell = \pi^k d v = d \pi^{k-1}v \in \B_\ell$.  But $\B_\ell \bot \H_\ell$ implies that $q_\ell = 0$, which is a contradiction.
\end{proof}


Combining the above two propositions yields the following lemma.
\begin{lemma}
\label{lem3}
$P_\ell:\H \rightarrow \H_\ell$ is an isomorphism.  In addition, $\|P_\ell^{-1}\| \le \|P_0^{-1}\| <\infty$.
\end{lemma}
\begin{proof}
Since $\beta = \beta_\ell<\infty$ and $P_\ell$ is injective, we conclude $P_\ell$ is an isomorphism. Consider the constant $c_0 = \inf_{q\in \H, \|q\| = 1}\|P_0 q\|$, which is positive because it is the infimum over a compact set of a continuous and positive function. Then $\|P_0^{-1}\|\leq c_0^{-1}$. By Proposition  \ref{prop3},  $\|P_0 q\|\leq \|P_1 q\| \leq \|P_2 q\| \leq ...$, and so we have $\|P_{\ell+1}^{-1}\| \le \|P_\ell^{-1}\| \leq ... \leq \|P_0^{-1} \| \leq c_0^{-1}$.
\end{proof}

\section{A posteriori error estimates}

\subsection{Previous estimates}
Our first goal is to control the error quantity $\E_\ell$ defined in \eqref{eq2-9}, which is a measure of the distance between $\H^k$ and $\H_\ell^k$.  

We follow \cite{DH14}, where a posteriori error estimates were given for measuring the gap between $\H^k$ and $\H_\ell^k$.  Let $\T_\ell$ be the mesh on level $\ell$, and let $h_T=|T|^{1/n}$ for $T \in \T_\ell$. Given $q_\ell \in \H_\ell^k$, let 
\begin{align}
\eta_\ell(T; q_\ell)=h_T\|\delta_{T} q_\ell\|_T + h_T^{1/2} \|\llbracket \trace \star q_\ell \rrbracket \|_{\partial T}
\end{align}
and
\begin{align} \eta_\ell(T)^2=\sum_{l=1}^{\beta} \eta_\ell(T; q_\ell^l)^2, ~~~\eta_\ell(\T) = \left ( \sum_{T \in \T} \eta_\ell(T)^2 \right ) ^{1/2}, ~\T \subset \T_\ell.
\end{align}
In the concrete case of the classical space given in \eqref{eq0-1}, we have $\eta_\ell(T; q_\ell) = h_T \|\Div q_\ell\|_T + h_T^{1/2} \|\llbracket q_\ell \cdot {\bf n} \rrbracket \|_{\partial T}.$  A slight modification of Lemma 9 and Lemma 13 of \cite{DH14} yields 
\begin{align}
\label{eq:practapost}
\eta(\T_\ell) \le C_1 \E_\ell \lesssim \eta(\T_\ell).
\end{align}
with $C_1$ and the constant hidden in $\lesssim$ depending on the shape regularity properties of $\T_0$, but independent of essential quantities including especially the dimension $\beta$ of $\H^k$ and $\H_\ell^k$.  From \cite{DH14} we also have
\begin{align}
{\rm gap}(\H^k, \H_\ell^k) \lesssim \eta(\T_\ell) \lesssim \sqrt{\beta} {\rm gap}(\H^k, \H_\ell^k).
\end{align}
Employing the error notion $\E_\ell$ thus allows us to obtain estimates with entirely nonessential constants.  

\subsection{Non-computable error estimators}

Convergence of adaptive FEM for multiple and clustered eigenvalues has been studied in \cite{DHZ15, Gal15, BD15}.  Our problem is similar in that our AFEM approximates a subspace rather than a single function.  The estimators defined above with respect to $\{q_\ell^j\}$ are problematic when viewed from the standpoint of standard proofs of AFEM contraction, which require continuity between estimators at adjacent mesh levels.  Because the bases $\{q_\ell^j\}$ and $\{q_{\ell+1}^j\}$ are not generally aligned, such continuity results are not meaningful.  

Following \cite{DHZ15}, we employ a non-computable intermediate estimator which solves this alignment problem and is equivalent to $\eta_\ell(T)$.  Let $\{q^j\}_{j=1}^{\beta}$ be a fixed orthonormal basis for $\H^k$.   We define 
\begin{align}
\mu_\ell(T)^2 = \sum_{j=1}^{\beta} \eta_\ell(T; P_\ell q^j)^2, ~~~\mu_\ell(\T) = \left ( \sum_{T \in \T} \mu_\ell(T)^2 \right ) ^{1/2}, ~\T \subset \T_\ell.
\end{align}

We next establish that approximation of $\H^k$ is sufficiently good on the initial mesh to guarantee that the estimators $\eta_\ell(T)$ and $\mu_\ell(T)$ are equivalent.  

\begin{theorem}\label{th:equivalence}
$$\mu_\ell(T) \le \|P_\ell\| \eta_\ell(T) \le  \|P_\ell\| \|P_\ell^{-1}\| \mu_\ell(T) \le \|P_{\underline{\ell}}^{-1}\| \mu_\ell(T), ~~ T \in \T_\ell, ~~0 \le \underline{\ell} \le \ell.$$ 
\end{theorem}
\begin{proof}
Recall that $\mu_\ell$ is defined using $\{P_\ell q^j \}$ with $\{q^j\}_{j=1}^{\beta}$ a fixed orthonormal basis for $\H^k$ and $\eta_\ell$ is defined using the orthonormal basis $\{q_\ell^j\}_{j=1}^{\beta_\ell}$.  Also, $\beta = \beta_\ell <\infty$. Let $\boldsymbol q = (q^1, q^2, \cdots, q^{\beta})^T$ and $\boldsymbol q_\ell = (q_\ell^1, q_\ell^2, \cdots, q_\ell^{\beta})^T$.  The matrix $[\langle q^i, q_\ell^j \rangle]=:M:\mathbb{R}^\beta \rightarrow \mathbb{R}^\beta$ satisfies $P_\ell\boldsymbol q = M \boldsymbol q_\ell$.  

Following the proof of \cite[Lemma 3.1]{BD15}, let $B:=MM^T=[\langle P_\ell q^i, P_\ell q^j \rangle]$.  $MM^T$ and $M^TM$ are isospectral and thus have the same $2$-norm, and $\|M^T\|^2=\|M\|^2=\|M^TM\|=\|B\|$.  We thus compute $\|M^T\|$.  Given $v \in \mathbb{R}^\beta$ with $|v|=1$, let $\tilde{v} = \sum_{j=1}^\beta v_i q^i$ so that $\tilde{v} \in \H$ with $\|\tilde{v}\|=1$.   Then $|M^T v|=|[\sum_{j=1}^\beta \langle q_\ell^i, v_j q^j \rangle] | =|[\langle q_\ell^i, \tilde{v} \rangle ]|= \|P_\ell \tilde{v}\| \le \|P_\ell\|$.  Thus $\|M\|=\|M^T\|\le \|P_\ell\|\le 1$.  Similarly, $\|M^{-1}\|^{-1} = \|M^{-T}\|  \ge \inf_{\tilde{v} \in \H, \|\tilde{v}\|=1} \|P_\ell \tilde{v}\| =\|P_\ell^{-1}\|^{-1}$.  Thus $\|M^{-1}\| \le \|P_\ell^{-1}\|$.


As $\delta$ is linear and commutes with $M$, we have
$$
\sum_{j=1}^{\beta} \|\delta_T P_\ell q^j\|_T^2 = \sum_{j=1}^{\beta} \|(M \delta_T \boldsymbol q_\ell)_j\|_T^2\leq \|M\|^2 \sum_{j=1}^{\beta} \|\delta_T q_\ell^j\|_T^2\leq \|P_\ell \|^2 \sum_{j=1}^{\beta} \|\delta_T q_\ell^j\|_T^2.
$$
Similarly 
$$
\sum_{l=1}^{\beta} \|\delta_T q_\ell^l\|_T^2 \leq \|M^{-1}\|^2 \sum_{l=1}^{\beta} \|\delta_T P_\ell q^l\|_T^2 \leq \|P_\ell^{-1}\|^2 \sum_{l=1}^{\beta} \|\delta_T P_\ell q^l\|_T^2.
$$
Similar inequalities hold for the boundary term. 
\end{proof}

Theorem \ref{th:equivalence} and \eqref{eq:practapost} immediately imply an a posteriori bound using $\mu_\ell$.

\begin{corollary}
\begin{align} 
\label{eq:apost}
\E_\ell^2 \le C_1 \|P_\ell^{-1}\|^2 \mu_\ell(\T_\ell)^2 \le C_1 \|P_{\underline{\ell}}^{-1}\|^2 \mu_\ell(\T_\ell)^2, ~~0 \le \underline{\ell} \le \ell,
\end{align}
where $C_1$ is independent of essential quantities.
\end{corollary}

\subsection{Localized a posteriori estimates} 
\label{sec:locapost}

As is common in AFEM optimality proofs, we require a localized upper bound for the difference between discrete solutions on nested meshes.  More precisely, let $\R_{\T_\ell \rightarrow \tilde{\T}}$ be the set of elements refined in passing from $\T_\ell$ to $\tilde{\T}$.  A standard estimate for elliptic problems with finite element solutions $u_\ell$ and $\tilde{u}$ on $\T_\ell, \tilde{\T}$ is $|||u_\ell -\tilde{u}||| \lesssim (\sum_{T \in \R_{\T_\ell \rightarrow \tilde{T}}} \xi(T)^2)^{1/2}$, where $\xi(T)$ is a standard elliptic residual indicator and $|||\cdot |||$ is the energy norm.  The estimate we prove is not as sharp but suffices for our purposes.  

\begin{lemma}
\label{lem:locapost}
Let $\T$ be a refinement of $\T_\ell$ so that $V_\ell^k \subset V_\T^k \subset H \Lambda^k$.  Then there exists a set $\hat{\R}_{\T_\ell \rightarrow \T}$ with $\R_{\T_\ell \rightarrow \T} \subset \hat{R}_{\T_\ell \rightarrow \T}$ and
\begin{align}
\label{eq6-100}
\# \hat{\R}_{\T_\ell \rightarrow \T} \lesssim \# \R_{\T_\ell \rightarrow \T}
\end{align}
such that 
\begin{align}
\label{eq6-101}
\sum_{j=1}^{\beta} \|(P_\ell -P_\T) q^j\|^2 \le C_2^2 \sum_{T \in \hat{\R}_{\T_\ell \rightarrow \T}} \eta_\ell(T)^2,
\end{align}
where $C_2$ is independent of essential quantities.
\end{lemma}

\begin{proof}  
Following the notation used in the proof of Theorem \ref{th:equivalence}, denote by ${\bf q}$, ${\bf q}_\ell$, and ${\bf q}_\T$ the column vectors of orthonormal basis functions for $\H$, $\H_\ell$, and $\H_\T$, respectively.  Let $M=(q^i, q_\T^j)$ be the matrix satisfying $P_\T {\bf q} = M {\bf q}_\T$; recall that $\|M\| \le \|P_\T\| \le 1$, with $\|M\|$ the matrix (operator) 2-norm.  Note also that $P_\ell {\bf q}  =P_\ell P_\T {\bf q}$, since $P_\ell {\bf q}  = P_{\Z_\ell} {\bf q}$, $P_\T {\bf q} = P_{\Z_\T} {\bf q}$, and $\Z_\ell \subset \Z_\T$ so that $P_{\Z_\ell} q_\T = P_\ell q_\T$, $q_\T \in \H_\T$.   We then compute
\begin{align} 
\label{eq6-110}
\begin{aligned}
\sum_{j=1}^{\beta} \|(P_\ell -P_\T) q^j\|^2 &  = \| | (P_\ell -P_\T) {\bf q}|\|^2 = \| | (P_\ell P_\T-P_\T) {\bf q}|\|^2
\\ & = \| |P_\ell M {\bf q}_\T - M {\bf q}_\T| \|^2 = \| | M(P_\ell {\bf q}_\T- {\bf q}_\T) | \|^2 
\\ & \le \|M\|^2  \| | P_\ell {\bf q}_\T-{\bf q}_\T|\|^2  \le \| | {\bf q}_\ell -P_\T {\bf q}_\ell |\|^2,
\end{aligned}
\end{align}
where in the last inequality above we employ $\|M\| \le 1$ along with \eqref{eq2-9-a}.  

Following \cite[(2.12) and Lemma 9]{DH14}, we note that $q_\ell^j -P_\T q_\ell^j \in \B_\T \perp \H_\T$, and $q_\ell^j \perp \B_\ell$.  Thus for $1 \le j \le \beta$ and some $v_\T \in V_\T^{k-1}$ with $\|v_\T\|_{H \Lambda(\Omega)} \simeq 1$, 
\begin{align}
\label{eq6-111}
\begin{aligned}
\|q_\ell^j - P_\T q_\ell^j\|  \lesssim \sup_{v \in V_\T^{k-1}, \|v\|_{H\Lambda(\Omega)}=1} \langle q_\ell^j-P_\T q_\ell^j, dv \rangle \lesssim \langle q_\ell^j-P_\T q_\ell^j, v_\T \rangle.
\end{aligned}
\end{align}

We next apply the regular decomposition result \eqref{reg_decomp} to find $z \in H^1 \Lambda^{k-1}(\Omega)$ with $dz = d v_\T$ (note that $\varphi$ as in \eqref{reg_decomp} plays no role here since $v_\T=d \varphi+ z$ implies $d v_\T = dz$).  We now denote by $\pi_\T$ and $\pi_{\ell}$ the Falk-Winther interpolants $\pi_{FW}$ on $\T$ and $\T_\ell$, respectively.  The commutativity of $\pi_\T$ implies that $d v_\T = \pi_\T d v_\T = d \pi_\T z$, so that using orthogonality properties as above we have
\begin{align}
\label{eq6-115}
\|q_\ell^j - P_\T q_\ell^j\| \lesssim \langle q_\ell^j, d \pi_\T z \rangle = \langle q_\ell^j, d (I-\pi_\ell) \pi_\T z \rangle.  
\end{align}
In addition, by \eqref{FW_interp3} we have for some $\hat{R}_{\T_\ell \rightarrow \T}$ satisfying \eqref{eq6-100}
\begin{align}
{\rm supp}( d(I-\pi_\ell) \pi_\T z ) \subset \cup_{T \in \hat{R}_{\T_\ell \rightarrow \T}} T.
\end{align}
Integrating by parts elementwise the last inner product in \eqref{eq6-115} and carrying out standard manipulations yields
\begin{align}
\label{eq6-114}
\begin{aligned}
\|q_\ell^j - P_\T q_\ell^j\| & \lesssim \sum_{T \in \hat{R}_{\T_\ell \rightarrow \T}} \eta_\ell (T; q_\ell^j) 
\\ & ~~~\cdot (h_T^{-1} \| (I-\pi_\ell) \pi_\T z\|_T + h_T^{-1/2} \|(I-\pi_\ell) \pi_\T z\|_{\partial T}).  
\end{aligned}
\end{align}

A standard scaled trace inequality may be applied to the term $\|d (I-\pi_\ell) \pi_\T z\|_{\partial T}$ only if extra care is taken.  We first write $(I-\Pi_\ell) \pi_\T z = [\pi_\T z - I_{SZ} z] + [I_{SZ} z - \pi_\ell I_{SZ} z] + [\pi_\ell (I_{SZ} z -\pi_\T z)]:= I+II+III$, where $I_{SZ}$ is the Scott-Zhang interpolant on $\T$.  For $T' \in \T$, we apply a standard scaled trace inequality $\|v\|_{L_2(\partial T')}^2 \lesssim h_{T'}^{-1} \|v\|_{T'} + h_T |v|_{H^1(T')}$, an inverse inequality, and the approximation properties \eqref{FW_interp4} and \eqref{SZ_interp} to find
\begin{align}
\label{eq6-119}
\begin{aligned}
h_T^{-1} \|I\|_{\partial T}^2 &  \lesssim \sum_{T' \in \T, T' \subset T} (h_T h_{T'})^{-1} \| I\|_{T'}^2
\\ & \lesssim \sum_{T' \in \T, T' \subset T} (h_T h_{T'})^{-1} (\|z-\pi_\T z\|_{T} + \|z-I_{SZ} z\|_{T})^2
\\ & \lesssim |z|_{H^1(\omega_T)}^2,
\end{aligned}
\end{align}
where we have used the fact that $h_{T'} \le h_T$ when $T' \subset T$ as above.  Next, $III \in V_\ell^{k-1}$, so we apply the trace inequality on $T \in \T_\ell$, the stability estimate \eqref{FW_interp2} and an inverse inequality, and then approximation properties as before to find
\begin{align}
\label{eq6-220}
\begin{aligned}
h_T^{-1} \|III\|_{\partial T}^2 & \lesssim h_T^{-2} \|\pi_\ell (\pi_\T - I_{SZ})z \|_{T}^2 
\\ & \lesssim \sum_{T' \subset T} (h_T h_{T'})^{-1} \|(\pi_\T-I_{SZ} z)\|_{\omega_{T'}}^2
\\ & \lesssim \sum_{T' \subset T} |z|_{H^1\Lambda (\omega_{T'}')}^2 \lesssim |z|_{H^1 \Lambda(\omega_{T}')}^2.
\end{aligned}
\end{align}
Here we let $\omega_T' = \cup_{\tilde{T} \in \omega_T} \omega_{\tilde{T}}$.  Finally, we have $II \in H^1\Lambda(T)$, $ T \in \T_\ell$, so we directly apply a scaled trace inequality, approximation properties of $\pi_\ell$, and $H^1$ stability of $I_{SZ}$ on $T$ to find that 
\begin{align}
\label{eq6-221}
h_T^{-1} \|III\|_{\partial T}^2 \lesssim |z|_{H^1\Lambda (\omega_T ')}^2.
\end{align}
Similar arguments yield 
\begin{align}
\label{eq6-222}
h_T^{-2} \|(I-\pi_\ell) \pi_\T z\|_{L_2(T)} \lesssim |z|_{H^1 \Lambda (\omega_T')}^2.
\end{align}
Inserting the previous inequalities into \eqref{eq6-114}, applying the Cauchy-Schwartz inequality over $\T_\ell$, employing finite overlap of the patches $\omega_T$, and finally using \eqref{reg_decomp} along with $\|v_\T\|_{H\Lambda(\Omega)} \lesssim 1$, we find that
\begin{align}
\label{eq6-223}
\begin{aligned}
\|q_\ell^j -P_\T q_\ell^j\| & \lesssim ( \sum_{T \in \hat{R}_{\T_\ell \rightarrow \T}}  \eta_\ell (T; q_\ell^j)^2)^{1/2} |z|_{H^1(\Omega)}
\\ & \lesssim ( \sum_{T \in \hat{R}_{\T_\ell \rightarrow \T}} \eta_\ell (T; q_\ell^j)^2)^{1/2} \|v\|_{H \Lambda (\Omega)}
\\ & \lesssim ( \sum_{T \in \hat{R}_{\T_\ell \rightarrow \T}} \eta_\ell (T; q_\ell^j)^2)^{1/2}.  
\end{aligned}
\end{align}

Summing over $j$ yields \eqref{eq6-101}.  
\end{proof}

\begin{remark}
In \cite{ZCSWX11}, the authors employ the local regular decomposition results of \cite{Sch08}, several different quasi-interpolants, and a discrete regular decomposition as in \cite{HX07} to prove a localized a posteriori upper bound for time-harmonic Maxwell's equations.  We do not need a local regular decomposition here, but rather combine the more powerful interpolant recently defined in \cite{FW14}, the Scott-Zhang interpolant, a global regular decomposition, and some ideas related to discrete regular decompositions as in \cite{HX07} in order to prove our local upper bound.   
\end{remark}

\begin{remark}
Proof of a posteriori bounds for harmonic forms is slightly simpler than for general problems of Hodge-Laplace type.  Our technique for proof of localized a posteriori bounds does however carry over to the more general case.  Proof of such upper bounds requires testing various terms with $v_\T \in V_\T$ and then subtracting $\pi_\ell v_\T$ via Galerkin orthogonality.  Employing a global regular decomposition yields $v_\T = d \varphi + z$, and subsequently $v_\T = \pi_\T v_\T = d \pi_\T \varphi + \pi_\T z$.   Subtracting off $\pi_\ell v_\T$ via Galerkin orthogonality yields $v_\T -\pi_\ell v_\T = d [(I-\pi_\ell) \pi_\T \varphi] + (I-\pi_\ell)\pi_\T z$. Proceeding as above using a combination of standard residual estimation tools and the Scott-Zhang interpolant will generically lead to localized upper bounds.  It is crucial that each of the terms $d [(I-\pi_\ell) \pi_\T \varphi]$ and $(I-\pi_\ell)\pi_\T z$ is individually locally supported.  The Falk-Winther interpolant plays a critical role as all of its major properties are employed in the proof.  
\end{remark}

\section{Contraction}
\label{sec:contraction}

We employ a standard adaptive finite element method of the form 
\begin{equation}
\textsf{solve} \rightarrow \textsf{estimate} \rightarrow \textsf{mark} \rightarrow \textsf{refine}.
\label{eq1-0}
\end{equation}
Our contraction proof follows the standard outline \cite{CKNS08} in that it combines an a posteriori error estimate, an estimator reduction property relying on properties of the marking scheme, and an orthogonality result in order to establish stepwise contraction of a properly defined error notion.

In the ``\textsf{mark}'' step we use a bulk (D\"orfler) marking.  That is, we fix $0<\theta \le 1$ and choose a minimal set $\M_\ell \subset \T_\ell$ such that
\begin{align}
\label{eq:bulkmarking}
\eta_\ell(\M_\ell) \ge \theta \eta_\ell(\T_\ell).
\end{align}
The following is an easy consequence of Theorem \ref{th:equivalence} and $\|P_\ell^{-1}\| \le \|P_0\|^{-1}$.  
\begin{proposition}
\label{prop:modmarking}
Let $0 < \theta \le 1$, and let $\M_{\underline{\ell}} \subset  \T_\ell$ satisfy \eqref{eq:bulkmarking}.  Let also $\bar{\theta} = \theta (\|P_\ell\| \|P_\ell^{-1}\|)^{-2}$ and $\tilde{\theta}_\ell= \theta \|P_\ell^{-1}\|^{-2}$.   Then for $\ell \ge 0$ and $0 \le \underline{\ell} \le \ell$, 
\begin{align}
\label{eq:modmarking}
\mu_\ell(\M_{\underline{\ell}})^2 \ge \bar{\theta}_\ell \mu_\ell(\T_\ell)^2  \ge \tilde{\theta}_{\underline{\ell}} \mu_\ell(\T_\ell)^2.
\end{align}
\end{proposition}

\begin{remark}
\label{rem:theta}
Note that $\|P_\ell\| \|P_\ell^{-1}\|=1$ when $\beta=1$, and for $\beta>1$ the deviation of $\|P_\ell\| \|P_\ell^{-1}\|$ from 1 is dependent on the isotropy of $P_\ell$.  Thus if $\|P_\ell q^i \| \approx \|P_\ell q^j\|$ for $1 \le i, j \le \beta$, then $\|P_\ell\| \|P_\ell^{-1}\| \approx 1$ even if $\H$ and $\H_\ell$ do not overlap strongly.  If $\beta=1$ or $P_\ell$ is isotropic, the theoretical and practical AFEM's will therefore mark the same elements for refinement.  
\end{remark}

We next establish continuity of the theoretical indicators (also known as an estimator reduction property). The proof is standard and is thus omitted.  
\begin{lemma}
Given $0 <\theta \le 1$, let $\M_{\underline{\ell}} \subset \T_\ell$ satisfy \eqref{eq:bulkmarking}.  Assume also that each $T \in \M_{\underline{\ell}}$ is bisected at least once in passing from $\T_\ell$ to $\T_{\ell+1}$.  Then there are constants $C_2$ and $1>\lambda>0$ independent of essential quantities such that for any $\alpha>0$, any $\ell \ge 0$, and any $0  \le \underline{\ell} \le \ell$,  
\begin{align}
\begin{aligned}
\label{eq:estred}
\mu_{\ell+1}(\T_{\ell+1})^2  & \le (1+ \alpha) (1-\lambda \bar{\theta}_{\ell}) \mu_\ell(\T_\ell)^2 + C_2 (1+\frac{1}{\alpha}) \sum_{j=1}^{\beta} \|(P_\ell-P_{\ell+1})q^j\|^2
\\  & \le (1+ \alpha) (1-\lambda \tilde{\theta}_{\underline{\ell}}) \mu_\ell(\T_\ell)^2 
 + C_2 (1+\frac{1}{\alpha}) \sum_{j=1}^{\beta} \|(P_\ell-P_{\ell+1})q^j\|^2.
\end{aligned}
\end{align}
Here $\bar{\theta}_{\ell}$ and $ \tilde{\theta}_{\underline{\ell}}$ are as defined in Proposition \ref{prop:modmarking}.  
\end{lemma}
%

Although $\H_\ell^k \not\subseteq \H_{\ell+1}^k$, the Hodge decomposition and the differential complex-conforming structure of the finite element spaces nonetheless yields the following essential orthogonality result. 
\begin{theorem} \label{lem1} 
For $q \in \H^k$, 
\begin{equation}\label{orth}
\|q-P_{\ell+1} q \|^2 = \|q-P_\ell q\|^2 - \|(P_\ell -P_{\ell+1})q\|^2.
\end{equation}
\end{theorem}
\begin{proof}
It suffices to prove $(q-P_{\ell+1} q, (P_\ell -P_{\ell+1})q) = 0$.  This is a consequence of $(P_\ell -P_{\ell+1})q \in \Z_{\ell+1}^k$, which holds due to the nestedness $\Z_\ell^k\subset \Z_{\ell+1}^k$ and the fact that $P_\ell : \H^k \rightarrow \H_\ell^k$ also acts as the $L_2$ projection from $\H^k$ to $\Z_\ell^k$.
\end{proof}

Assembling the above estimates yields the following contraction result. The proof is standard, except we must track dependence of constants on the mesh level $\ell$.
\begin{theorem}  \label{theorem:contraction}
For each $\underline{\ell}>0$, there exist $0<\rho_{\underline{\ell}}<1$ and $\gamma_{\underline{\ell}} >0 $ such that for $\ell \ge \underline{\ell}$, 
\begin{align}
\label{eq:contraction}
\E_{\ell+1}^2 + \gamma_{\underline{\ell}} \mu_{\ell+1}(\T_{\ell+1})^2 \le \rho_{\underline{\ell}} \left (\E_\ell^2 + \gamma_{\underline{\ell}} \mu_\ell(\T_\ell)^2 \right ).
\end{align}
Here $1>\rho_{\underline{0}} \ge \rho_{\underline{1}} \ge ... \ge \underline{\rho}:=\lim_{\underline{\ell} \rightarrow \infty} \rho_{\underline{\ell}} >0$.  $\rho_{\underline{\ell}}$ depends on $\|P_{\underline{\ell}}^{-1}\|$ but not on other essential quantities, and $\underline{\rho}$ is independent of essential quantities.  Finally, $0<\gamma_{\underline{\ell}} < C_2^{-1}$ with $C_2$ as in \eqref{eq:estred}.  
\end{theorem}

\begin{proof}
Given $\underline{\ell} \ge 0$ and $\alpha$ as in \eqref{eq:estred}, let $\gamma_{\underline{\ell}}=\frac{1}{C_2(1+\alpha^{-1})}$.  (Here we suppress the dependence of $\alpha$ on $\underline{\ell}$.)  Then $0<\gamma_{\underline{\ell}} < C_2^{-1}$, as asserted.    Combining \eqref{eq:estred} with \eqref{orth} then yields
\begin{align}
\label{eq4-100}
\E_{\ell+1}^2 +  \gamma_{\underline{\ell}} \mu_{\ell+1}^2 \le \E_\ell^2 + \gamma_{\underline{\ell}} (1+\alpha)(1-\lambda \tilde{\theta}_{\underline{\ell}})\mu_\ell^2.
\end{align}
Let now $0<\rho_{\underline{\ell}}<1$.  From \eqref{eq:apost} we have $(1-\rho_{\underline{\ell}})\E_\ell^2 \le C_1 (1-\rho_{\underline{\ell}}) \|P_{\underline{\ell}}^{-1}\|^2 \mu_\ell^2$, so that
\begin{align}
\label{eq4-101}
\E_{\ell+1}^2 + \gamma_{\underline{\ell}} \mu_{\ell+1}^2 \le \rho_{\underline{\ell}} \E_\ell^2 + \gamma_{\underline{\ell}} \left  [ (1+\alpha)(1-\lambda \tilde{\theta}_{\underline{\ell}}) + \gamma_{\underline{\ell}}^{-1} C_1(1-\rho_{\underline{\ell}}) \|P_{\underline{\ell}}^{-1}\|^2 \right ] \mu_\ell^2.
\end{align}
We next set $\rho_{\underline{\ell}} =  (1+\alpha)(1-\lambda \tilde{\theta}_{\underline{\ell}}) + \gamma_{\underline{\ell}}^{-1} C_1(1-\rho_{\underline{\ell}}) \|P_{\underline{\ell}}^{-1}\|^2$ and solve for $\rho_{\underline{\ell}}$.  Before doing so, we introduce the shorthand $K_{\underline{\ell}} = 1-\lambda \tilde{\theta}_{\underline{\ell}}$ and $M_{\underline{\ell}} = C_1 C_2 \|P_{\underline{\ell}}^{-1}\|^2$.  Then
\begin{align}
\label{eq4-102}
\rho_{\underline{\ell}} = \frac{\alpha^2 K_{\underline{\ell}} + \alpha (K_{\underline{\ell}} + M_{\underline{\ell}}) + M_{\underline{\ell}}}{\alpha (1+ M_{\underline{\ell}}) + M_{\underline{\ell}}}.
\end{align}
Recalling that $\alpha>0$ is arbitrary, we minimize the above expression with respect to $\alpha$ to obtain
\begin{align}
\label{eq4-103}
\rho_{\underline{\ell}} = \frac{2 \sqrt{M_{\underline{\ell}} K_{\underline{\ell}} (1+ M_{\underline{\ell}} -K_{\underline{\ell}})} + M_{\underline{\ell}}^2 + M_{\underline{\ell}} + K_{\underline{\ell}} - K_{\underline{\ell}} M_{\underline{\ell}}}{(1+M_{\underline{\ell}})^2}.   
\end{align}

We now analyze the dependence of $\rho_{\underline{\ell}}$ on $\ell$.  Recall that $M_{\underline{\ell}} = C_1 C_2 \|P_{\underline{\ell}}^{-1}\|^2$ and  $K_{\underline{\ell}} = 1-\lambda \theta \|P_{\underline{\ell}}^{-1}\|^{-2}$ are decreasing in $\underline{\ell}$.  Tedious but elementary calculations also show that for $0<K_{\underline{\ell}} <1$ and $M_{\underline{\ell}}>0$, $\rho_{\underline{\ell}}$ is increasing in both $M_{\underline{\ell}}$ and $K_{\underline{\ell}}$.  Thus \eqref{eq:contraction} holds for all $\ell \ge 0$ with $\rho_{\underline{\ell}} = \rho_{\underline{0}}$.  We in turn see that $\|P_{\underline{\ell}}\|, \|P_{\underline{\ell}}^{-1}\| \rightarrow 1$ as $\underline{\ell} \rightarrow \infty$, so $M_{\underline{\ell}} \downarrow C_1 C_2 :=M_\infty$.  In addition, $K_{\underline{\ell}} \downarrow 1-\lambda \theta:=K_\infty$ as $\ell \rightarrow \infty$.   Thus $\rho_{\underline{\ell}} \le \rho_{\underline{0}}$ and $\rho_{\underline{\ell}}$ decreases to 
\begin{align}
\underline{\rho} = \frac{ 2 \sqrt{C_1C_1 (1-\lambda \theta)( C_1 C_2 + \lambda \theta)} + C_1^2 C_2^2 + C_1 C_2+(1-\lambda \theta)(1-C_1 C_2)} { (1+C_1 C_2)^2}
\end{align}
as $\underline{\ell} \rightarrow \infty$.  This completes the proof.  
\end{proof}  

\begin{remark}  In Remark \ref{rem:theta} we established that the theoretical and practical AFEM mark the same elements for refinement when $P_\ell$ is isotropic, and in particular when $\beta=1$.  We could sharpen the proof of Theorem \ref{theorem:contraction} to take advantage of this fact by redefining $K_{\underline{\ell}} = 1-\lambda \bar{\theta}_{\underline{\ell}} = 1-\lambda \theta (\|P_{\underline{\ell}}^{-1}\| \|P_{\underline{\ell}}\|)^{-2}$.  However, doing so would compromise monotonicity of the sequence $\{\rho_{\underline{\ell}} \}$, and $M_{\underline{\ell}}$ depends on $\|P_{\underline{\ell}}^{-1}\|$ and not on the product $\|P_{\underline{\ell}}^{-1}\| \|P_{\underline{\ell}}\|$ in any case.  
\end{remark}

\begin{remark} 
Dependence of $\rho_{\underline{\ell}}$ on $\|P_{\underline{\ell}}^{-1}\|$ seems unavoidable in our proofs.    We prove a contraction by combining the orthogonality relationship \eqref{orth}, the estimator reduction inequality \eqref{eq:estred}, and the a posteriori estimate \eqref{eq:apost} in a canonical way \cite{CKNS08}.  The orthogonality relationship \eqref{orth} indicates that the error $\E_\ell^2$ is reduced by $\sum_{j=1}^\beta \|(P_{\ell+1}-P_\ell) q^j\|^2$ at each step of the AFEM algorithm.  This quantity is directly related to the theoretical indicators $\mu_\ell$ in the estimator reduction inequality \eqref{eq:estred}.  Combining these relationships leads to an error reduction on the order of $\mu_\ell(\T_\ell)$.  On the other hand, $\E_\ell$ is uniformly equivalent to the practical estimator $\eta_\ell$.  Because the theoretical and practical estimators are related by $\|P_\ell^{-1}\|^{-1} \eta_\ell \le \mu_\ell \le \|P_\ell\| \mu_\ell$, reducing $\E_\ell$ by $\mu_\ell(\T_\ell)$ is equivalent to a reduction lying between $\frac{1}{\|P_\ell^{-1}\|} \E_\ell$ and $\|P_\ell\|\E_\ell$. 
\end{remark}

\begin{remark}  
For elliptic source problems, a contraction is obtained from the initial mesh with contraction factor independent of essential quantities \cite{CKNS08}.  On the other hand, for elliptic eigenvalue problems such a contraction result holds only if the initial mesh is sufficiently fine \cite{Gal15, BD15}.  The situation for harmonic forms is intermediate between those encountered in source problems and eigenvalue problems.  No initial mesh fineness assumption is needed to guarantee a contraction, but we only show that the contraction constant is independent of essential quantities on sufficiently fine meshes.   In the case of eigenvalue problems a similar transition state likely exists in which AFEM can be proved to be contractive, but with contraction constant improving as resolution of the target invariant space improves.  
 \end{remark}

\section{Quasi-Optimality}
\label{sec:optimality}

Our proof of quasi-optimality follows a more or less standard outline, simplified somewhat by lack of data oscillation but made more complicated by improvement in the contraction factor as the mesh is refined.

\subsection{Approximation classes}

Given rate $s>0$ and ${\bf r} \in [\H]^\beta$, we let
\begin{align}
\label{eq6-120}
|{\bf r}|_{\A_s} = \sup_{N>0} \inf_{\# \T - \# \T_0 \le N} N^{-s} \left ( \sum_{j=1}^\beta \|r^j-P_\T r^j\|^2 \right ) ^{1/2}.
\end{align} 
$\A_s$ is then the class of all ${\bf r} \in [\H]^\beta$ such that $|{\bf r}|_{\A_s}<\infty$. 

If we applied $\A_s$ to arbitrary ${\bf r} \in [H\Lambda^k(\Omega)]^\beta$, it would be natural to replace the projection $P_\T$ onto $\H_T$ used in \eqref{eq6-120} by the $L_2$ projection onto the full discrete space $V_\T$.  We show in Proposition \ref{lem:equiv} that best approximation over the full finite element space is equivalent up to a constant to best approximation over $\H_\T$ only.  Thus our definition of $\A_s$ makes use of the full approximation strength of the finite element space $V_\T$, even though at first glance it may not seem that this is the case.  
\begin{proposition}
\label{lem:equiv}
Given $\T \in \mathbb{T}$ and $q \in \H^k$, 
\begin{align}
\label{eq6-121}
\|q-P_\T q\| \lesssim \inf_{q_\T \in V_\T} \|q-q_\T\|.
\end{align}
\end{proposition}
\begin{proof}
From (27) of \cite{AFW10}, we have $\|q-P_{\H_\T} q\| \le \|(I-\pi_\T) q\|$ with $\pi_\T$ a commuting cochain projection.  Thus for $q \in \H$, $q_\T \in V_\T^k$, we may use \eqref{CW_interp} to obtain
\begin{align}
\|q-P_\T q \| \le \|(q-q_\T) -\pi_{CW} (q-q_\T)\| \le (1+C)\|q-q_\T\|,
\end{align}
where $C$ is the $L_2$ stability constant for $\pi_{CW}$.  
\end{proof}

\subsection{Rate optimality}

We first state and prove two lemmas which are more or less standard in this context.  It is important, however, that the constants in these two lemmas are entirely independent of essential quantities.  

\begin{lemma}
\label{lem:dorf}
Let $C_1$ and $C_2$ be the constants from \eqref{eq:practapost} and \eqref{eq6-101}, respectively, and assume that $\theta < \frac{1}{C_1 C_2}$.  Then for $\T_\ell \subset \T$ with 
\begin{align}
\label{eq6-105}
\|{\bf q}-P_\T {\bf q}\| \le [1-\theta C_1 C_2 ] \|{\bf q} - P_\ell {\bf q}\|,
\end{align}
there holds that 
\begin{align}
\label{eq6-106}
\eta_\ell(\hat{R}_{\T_\ell \rightarrow \T}) \ge \theta \eta_\ell(\T_\ell).
\end{align}
\end{lemma}
\begin{proof}
Employing in turn \eqref{eq:practapost}, \eqref{eq6-105}, the triangle inequality, and \eqref{eq6-101} yields
\begin{align}
\begin{aligned}
\theta C_2 \eta_\ell(\T_\ell) & \le \theta C_1 C_2 \|{\bf q} - P_\ell {\bf q}\|
\\ & \le \|{\bf q} - P_\ell { \bf q} \| - \|{\bf q} - P_\T {\bf q}\|
\\ & \le \|P_\ell {\bf q} - P_\T {\bf q} \|
\\ & \le C_2 \eta_\ell (\hat{R}_{\T_\ell \rightarrow \T}).
\end{aligned}
\end{align}
Dividing through by $C_2$ completes the proof.  
\end{proof}

\begin{lemma}
The collection of marked elements $\M_\ell \subset \T_\ell$ defined by the marking strategy \eqref{eq:bulkmarking} satisfies
\begin{align}
\label{eq6-123}
\# \M_\ell \lesssim |{\bf q}|_{\A_s}^{1/s}  \E_\ell^{-1/s}.
\end{align}
\end{lemma}
\begin{proof}
By definition of ${\mathcal A}^s$ there exists a partition $\T' \in \mathbb{T}$ such that
\begin{equation} \label{6-124}
\# \T'-\# \T_0 \lesssim |{\bf q}|_{\A_s}^{1/s}\left [ ( 1-\theta C_1C_2) \E_\ell \right]^{-1/s}
\end{equation}
and
$$
\|{\bf q} - P_{\T'} {\bf q}\| \leq (1-\theta C_1 C_2) \E_\ell.
$$
The smallest common refinement $\T$ of $\T_\ell$ and $\T'$ is in $\mathbb{T}$ with $\# \T -\#\T_\ell \leq \# \T'-\#\T_0$ (cf. \cite{Ste07} last lines of the proof of Lemma 5.2).
Since $V_{\T'} \subset V_\T$, \eqref{eq6-121} and the last equation above yield
\begin{align*}
\|{\bf q}-P_{V_\T} {\bf q} \| \leq \|{\bf q} - P_{V_{\T'}} {\bf q}\|  \leq (1-\theta C_1 C_2)\E_\ell.  
\end{align*}
Thus $\eta(\hat{\R}_{\T_\ell \rightarrow \T}) \geq \theta \eta_\ell$ by Lemma \ref{lem:dorf}. Since $\M_\ell$ is the {\em smallest} subset of $\T_\ell$ with $\eta(\M_{\underline{\ell}}) \geq \theta \eta_\ell$, we conclude that
$$
\# \M_\ell \leq \# \hat{\R}_{\T_\ell \rightarrow \T} \lesssim \# \T -\#\T_\ell \lesssim \# \T' -\# \T_0.
$$
\end{proof}

We finally state our optimality result.
\begin{theorem}
\label{th:optimality}
Assume as in Lemma \ref{lem:dorf} that $\theta<\frac{1}{C_1 C_2}$, and assume that ${\bf q} \in \A_s$.  Given $\underline{\ell} \ge 0$, there exists a constant $C_{\underline{\ell}}$ depending on $\|P_{\underline{\ell}}^{-1}\|$ and the constant $C_{ref,\underline{\ell}}$ from \eqref{numrefined} but independent of other essential quantities such that
\begin{align}
\label{optimality}
\E_\ell \le C_{\underline{\ell}} (\# \T_\ell -\# \T_{\underline{\ell}})^{-1/s} |{\bf q}|_{\A_s}, ~~\ell \ge \underline{\ell}.
\end{align}
\end{theorem}
\begin{proof} We first compute using Theorem \ref{th:equivalence}, \eqref{eq:practapost}, and the fact that $\gamma_{\underline{\ell}} \le C_2^{-1}$ that for $k \ge \underline{\ell}$, 
\begin{align}
\label{eq6-150}
\E_k^2 + \gamma_{\underline{\ell}} \mu_k(\T_k)^2 \le \E_k^2 + \gamma_{\underline{\ell}} \eta_k(\T_k)^2 \lesssim (1+\gamma_{\underline{\ell}}) \E_k^2 \lesssim \E_k^2.
\end{align}
Thus
\begin{align}
\label{eq6-151}
\E_k^{-1/s} \lesssim (\E_k^2 + \gamma_{\underline{\ell}} \mu_k(\T_k)^2)^{-1/2s}.
\end{align}
We then use \eqref{eq:contraction} to obtain
\begin{align}
\label{eq6-152}
\E_k^{-1/s} \lesssim \rho_{\underline{\ell}}^{(\ell-k)/2s}  (\E_{\ell}^2+\gamma_{\underline{\ell}} \mu_{\ell}(\T_{\underline{\ell}})^2)^{-1/2s}.
\end{align}
From \eqref{numrefined}, it then follows that 
\begin{align}
\begin{aligned}
\# \T_\ell -\# \T_{\underline{\ell}} & = \sum_{k=\underline{\ell}}^{\ell-1} \# \R_{\T_{k} \rightarrow \T_{k+1}}
\le C_{ref, \underline{\ell}} \sum_{k=\underline{\ell}}^{\ell-1} \# \M_k 
 \lesssim C_{ref, \underline{\ell}} |{\bf q}|_{\A_s}^{1/s}  \sum_{k=\underline{\ell}}^{\ell-1} \E_k^{-1/s}
\\ &  \le C_{ref, \underline{\ell}}  |{\bf q}|_{\A_s}^{1/s} \sum_{k=\underline{\ell}}^{\ell-1} \rho_{\underline{\ell}}^{(\ell-k)/2s} (\E_{\ell}^2+\gamma_{\underline{\ell}} \mu_{\ell}(\T_{\underline{\ell}})^2)^{-1/2s}
\\ & \le C_{ref, \underline{\ell}}  \left ( 1-\rho_{\underline{\ell}}^{1/2s} \right )^{-1}  |{\bf q}|_{\A_s}^{1/s}  ( \E_{\ell}^2+\gamma_{\underline{\ell}} \mu_{\ell}(\T_{\underline{\ell}})^2)^{-1/2s}.
\end{aligned}
\end{align}
Setting  $C_{\underline{\ell}}:=C C_{ref, \underline{\ell}} \left ( 1-\rho_{\underline{\ell}}^{1/2s} \right )^{-1}$ and rearranging the above expression completes the proof.  
\end{proof}

\begin{remark}  As is standard in AFEM optimality results, $\theta$ is required to be sufficiently small in order to ensure optimality.  $\theta$ must however only be sufficiently small with respect to $\frac{1}{C_1 C_2}$, which is entirely independent of the dimension $\beta$ of $\H$, $\|P_\ell^{-1}\|$, and other essential quantities.  In contrast to elliptic eigenvalue problems \cite{Gal15, BD15}, we do not require an initial fineness assumption on $\T_0$ in order to guarantee that the threshold value for $\theta$ is independent of essential quantities.  

The constant $C_{\underline{\ell}}$ does however depend on $\|P_{\underline{\ell}}^{-1}\|$, and it is not clear that this constant will improve (decrease) as the mesh $\underline{\ell}$ increases even though $\|P_{\underline{\ell}}^{-1}\| \rightarrow 1$.  The factor $(1-\rho_{\underline{\ell}}^{1/2s})^{-1}$ is nonincreasing, but the factor $C_{ref, \underline{\ell}}$ arising from \eqref{numrefined} is not uniformly bounded in $\underline{\ell}$.  According to \cite[Theorem 6.1]{Ste08} it may in essence depend on the degree of quasi-uniformity of the mesh $\T_{\underline{\ell}}$ and thus may degenerate as the mesh is refined.  In order to guarantee a uniform constant, we apply Theorem \ref{th:optimality} with $\underline{\ell}=0$ and obtain a constant $C_{\underline{0}}$ which depends on $\|P_0^{-1}\|$ in addition to $\T_0$.

\end{remark}

\section{Extensions}
\label{sec:extensions}

In this section we briefly discuss possible extensions of our work.

\subsection{Essential boundary conditions.}  Many of our results extend directly to the case of essential boundary conditions in which the requirement $\trace \star q=0$ in \eqref{eq2-1} is replaced by $\trace q =0$, with the latter condition imposed directly on the finite element spaces.  The major hurdle in obtaining an immediate extension is the availability of quasi-interpolants which possess the necessary properties.  The Christiansen-Winther interpolant in Section \ref{sec:interpolants} has been defined and analyzed also for essential boundary conditions, while the Falk-Winther interpolant was fully analyzed in \cite{FW14} only for natural boundary conditions.  We used the properties of the Falk-Winther interpolant only to obtain the localized a posteriori upper bound (Lemma \ref{lem:locapost}), which is necessary to obtain a quasi-optimality result but not a contraction.  The contraction result given in Theorem \ref{theorem:contraction} thus extends immediately to the case of essential boundary conditions.  There is indication given in \cite{FW14} that the properties of the Falk-Winther interplant transfer naturally to essential boundary conditions, in particular by simply setting boundary degrees of freedom to 0.  Assuming this extension our quasi-optimality results also hold for homogeneous essential boundary conditions.  A suitable interpolant is also defined and analyzed in the paper \cite{ZCSWX11} for the practically important case $d=3$, $k=1$.  

\subsection{Harmonic forms with coefficients}  In electromagnetic applications the magnetic permeability $A$ is a symmetric, uniformly positive definite matrix having entries in $L_\infty(\Omega)$.  If $A$ is nonconstant, then the space
\begin{align}
\label{eq700}
\H_A(\Omega) =\{ {\bf v} \in L_2(\Omega)^3 : \curl {\bf v} = 0, ~ \Div A {\bf v} = 0, ~ A{\bf v} \cdot {\bf n} = 0 \hbox{ on } \partial \Omega \}
\end{align}
is the natural space of harmonic forms, but differs from that in \eqref{eq0-1}.  It is similarly possible to modify the more general definition \eqref{eq2-1} to include coefficients.  As is pointed out in \cite[Section 6.1]{AFW10}, the finite element exterior calculus framework applies essentially verbatim to this situation once the inner products used in all of the relevant definitions are modified to include coefficients.  Our a posteriori estimates and the contraction result of Section \ref{sec:contraction} similarly apply with minimal modification.  Extension of the optimality results of Section \ref{sec:optimality} is possible but complicated by the presence of oscillation of the coefficient $A$ in the analysis; cf. \cite{BD15, CKNS08, DHZ15}.  

\subsection{Approximation of cohomologous forms}  \label{sec7-3}  In some applications it is of interest to compute $P_\H f$ for a given (non-harmonic) form $f$.  This is for example the case in the finite element exterior calculus framework for solving problems of Hodge-Laplace type.  There $f$ is the right-hand-side data.  Because the Hodge-Laplace problem is only solvable for data orthogonal to $\H$, it is necessary to compute $P_\H f$ and solve the resulting system with data $f-P_\H f$; cf. \cite{HiKaWaWa11} for other applications. A possible adaptive approach to this problem is to adaptively reduce the defect between $\H$ and $\H_\ell$ as we do above and then project $f$ onto $\H_\ell$.  However, this approach requires computation of a multidimensional space, while the original problem only requires computation of a one-dimensional space (that spanned by $P_\H f$).  An alternate method would be to approximate the full Hodge decomposition.   More precisely, one could compute $P_{\B} f$ and $P_{\Z^\perp} f$ by solving two constrained elliptic problems, but expense is an obvious disadvantage of this method also.  

It may be desirable to instead adaptively compute $P_\H f$ directly.  It might for example be the case that some members of $\H_\ell$ have singularities which are not shared by $P_\H f$ (such situations arise in eigenfunction computations).  The task of constructing an AFEM for approximating only $P_\H f$ appears difficult, however.  Assume for the sake of argument the extreme case where $P_\H f \neq 0$, but $f \perp \H_\ell$.  The indirect approach of first controlling the defect between $\H$ and $\H_\ell$ and then projecting $f$ would continue to function with no problems in this case.  On the other hand, it is not clear how to directly construct an a posteriori estimate for $\|P_\H f - P_\ell f\|$ which would be nonzero when $\H_\ell \perp f$.  In particular, it is not difficult to construct an estimator for $\|P_\ell f-P_\H P_\ell f\|$ (cf. \cite{DH14}), but such an estimator would be 0 and thus not reliable in this case.  

\subsection{Alternate methods for computing harmonic forms}  \label{subsec:cutting}  We have largely ignored the actual method for obtaining $\H_\ell$ by simply assuming that we in some fashion produced an orthonormal basis for this space.  This viewpoint is consistent with the finite element exterior calculus framework that we have largely followed.  It also fits well with eigenvalue or SVD-based methods for computing $\H_\ell$, which are general with respect to space dimension and form degree and which may be easily implemented using standard linear algebra libraries.  Discussion of methods for producing such a basis consistent with the FEEC framework may be found in \cite{HiKaWaWa11}.  However, the process of producing $\H_\ell$ is in and of itself not entirely straightforward, and the method for producing it may affect the structure and properties of the resulting adaptive algorithm.  Different methods with potentially advantageous properties have been explored especially in three space dimensions \cite{DST09, RBGV13}.  The method of cutting surfaces is such an example.

Let $\Omega \subset \mathbb{R}^3$, and assume that there exist $\beta$ regular and nonintersecting cuts (two-dimensional hypersurfaces) $\sigma_j$ such that $\Omega_0=\Omega \setminus \cup_{j=1}^\beta \sigma_j$ is simply connected.  The assumption that such cuts exist is nontrivial with respect to domain topology and excludes for example the complement of a trefoil knot in a box \cite{BFG12}.  The methodology we discuss here thus does not apply in all situations where our theory above applies.  A domain for which such a set of cutting surfaces exist is called a Helmholtz domain.  Determination of cutting surfaces on finite element meshes is also a nontrivial and possibly computationally expensive problem \cite{DST09}, although in an adaptive setting one could potentially compute the cutting surfaces at low cost on a coarse mesh and transfer them to subsequent refinements.  

Assuming that $\Omega$ is Helmholtz, let $\varphi^j$ solve
\begin{align}
\begin{aligned}
-\Delta \varphi^j& =0 \hbox{ in } \Omega_0, ~~ \nabla \varphi^j \cdot {\bf n}=0 \hbox{ on } \partial \Omega,
\\ & \llbracket \varphi^j \rrbracket_{\sigma_i}= \delta_{ij} \hbox{ and }  \llbracket \nabla \varphi_j \cdot {\bf n}_i\rrbracket_{\sigma_i}=0, ~~1 \le i \le \beta.
\end{aligned}
\end{align}
The set $\{\nabla \varphi_j\}_{j=1}^\beta$ then serves as a basis for the space $\H^1$ of vector fields defined in \eqref{eq0-1}.  That is, the harmonic fields may be reclaimed from potentials consisting of $H^1$ functions.  

Next assume that the cuts $\sigma_j$ each consist of unions of faces in a simplicial mesh $\T$.  Let now $V^0$ be piecewise linear Lagrange elements on $\Omega$, and let $V_j^0$ be the set of functions which are continuous and piecewise linear in $\Omega \setminus \sigma_j$ and which satisfy $\llbracket v_h \rrbracket _{\sigma_j} =1$.  The canonical finite element approximation to $\varphi^j$ is to find $\varphi_\T^j \in V_j^0$ such that $\int_{\Omega \setminus \sigma_j} \nabla \varphi_\T^j \nabla v=0$ for all $v \in V^0$.  $\varphi_\T^j$ is only unique up to a constant, but since we are only interested in $\nabla \varphi_\T^j$ this has little effect on our discussion.  One can also verify that $\nabla \varphi_\T^j$ is a discrete harmonic field lying in the lowest-degree N\'ed\'elec edge space.  Thus as in the continuous case, the discrete harmonic fields can be recovered from potentials.  The same procedure may be applied with other complex-conforming pairs of spaces.  An analogous procedure also exists for essential boundary conditions.  

An obvious adaptive procedure for approximating $\H^1$ is to adaptively compute $\varphi_\T^j$, $j=1,...,J$, using standard AFEM for scalar Laplace problems.  Convergence and optimality follows from standard results such as those found in \cite{CKNS08} with slight modification to account for boundary conditions.    The basis $\{\nabla \varphi^j\}_{j=1}^\beta$ that we thus approximate is not orthonormal, but has the advantage of being fixed using criterion that are passed on to the discrete approximations.  Also note that here the approximation map $\varphi^j \rightarrow \varphi_\T^j$ is clearly linear.  If we however orthonormalized the vectors $\{\nabla \varphi_\T^j\}$ in order to produce $\{q_\T^j\}$ as in our previous assumptions, then the approximation $\{\nabla \varphi^j\} \rightarrow \{q_\T^j\}$ would be nonlinear.  This discussion indicates that while producing an orthonormal basis of forms is a nonlinear procedure, the nonlinearity is mild and not intrinsic to the task of finding {\it some} basis for the space of harmonic forms.

\section{Numerical experiments}
\label{sec:numerics}  

In this section we give brief numerical experiments which illustrate our theoretical results.  We employed the MATLAB-based finite element toolbox iFEM \cite{Ch09PP}.  Harmonic fields for two-dimensional domains were computed by first constructing the full system matrix for the Hodge (vector) Laplacian and then finding its kernel using the MATLAB \textsf{svds} command.  This approach may be relatively inefficient from a computational standpoint, but allowed accurate computation of the discrete harmonic basis using off-the-shelf components.  A more sophisticated algorithm for computing discrete harmonic fields, motivated by algebraic topology but only formulated for lowest-degree (Whitney) finite element forms, is available in \cite{RBGV13}.  

We adaptively computed the space $\H^1$ of harmonic forms on domains $\Omega \subset \mathbb{R}^2$.   The de Rham complex in two space dimensions may be realized as
\begin{align}
\label{eq800}
H^1(\Omega) \overset{\nabla}{\rightarrow} H(\rot;\Omega) \overset{\rot}{\rightarrow} L_2(\Omega),
\end{align}
where the rotation operator is given by $\rot {\bf v} = \partial_y v_1-\partial_x v_2$.  The adjoint of $d=\rot$ is the two-dimensional curl operator $\delta  = \curl \varphi=(-\partial_y, \partial_x)$.  The corresponding space $\H^1$ of harmonic forms consists of rotation- and divergence-free vector fields with vanishing normal component ${\bf v} \cdot \vec{n}$ on $\partial \Omega$.  

Our computations were performed on (non-simply connected) polygonal domains.  We briefly recall some standard facts about singularities of solutions to elliptic PDE on polygonal domains.  Assume that $-\Delta u =f$ with homogenous Neumann boundary conditions, that is, assume that $u$ solves a 0-Hodge Laplace problem for the complex \eqref{eq800}.  At a given vertex $v_i \in \partial \Omega$ with interior opening angle $\omega_i$, there generally holds $u(x) \sim r_i(x)^{\pi/\omega_i}$ with $r_i={\rm dist}(x,v_i)$.  Solutions to the 1-Hodge Laplace problem $(d \delta + \delta d) u =(-\nabla \Div + \curl \rot) u =f$ may generically be expected to have singularities one exponent stronger, that is, of the form $r_i^{\pi/\omega_i-1}$; cf. \cite{dauge_web, CD00} for related discussion of Maxwell's equations in three space dimensions.  

Because harmonic forms satisfy $(-\nabla \Div + \curl \rot) u =0$, one would expect a singularity structure similar to that for the corresponding Hodge Laplace problem.  For $q^j \in H^1$, we thus expect
\begin{align}
\label{eq801}
q^j \sim r_i^{\pi/\omega_i-1}
\end{align}
Another way to see this is to consider the Hodge decomposition $\vec{v}=\nabla \varphi + q + \vec{z}$  of a smooth vector field $\vec{v}$.  $\nabla \varphi$ solves $-\Delta \varphi=\Div \vec{v}$, so we generally expect $\nabla \varphi \sim r_i^{\pi/\omega_i-1}$ near $v_i$.  Because $\vec{v}$ is smooth, $q \in \H^1$ and $z \in \Z^{\perp}$ may be expected to have offsetting singularities.  As a final note, if $\omega_i> \pi$ (that is, if the opening at $v_i$ is nonconvex), then the exponent of $r_i$ in \eqref{eq801} is negative and thus we expect $q^j$ to be unbounded near $r_i$.   

In all of our experiments below we compute harmonic forms on polygonal domains in $\mathbb{R}^2$ using rotated Raviart-Thomas elements of lowest degree.  These elements give an a priori convergence rate of order $O(h)$ for $\|q-P_\ell q\|_{L_2(\Omega)}$ assuming sufficiently smooth $q$.  When $q$ has vertex singularities such as those described above, AFEM is able to recover a convergence rate of $(\# \T_\ell -\# \T_0)^{1/2}$, just as for standard piecewise linear Lagrange AFEM for computing gradients of solutions to $-\Delta u =f$.  

\subsection{Experiment 1:  $\beta=1$.}  Our goals in this experiment are to demonstrate improved convergence rates for AFEM vs. quasi-uniform mesh refinement and also to confirm that harmonic forms blow up at reentrant corners, as predicted by \eqref{eq801}.  Here $\Omega$ is a simple square annulus with reentrant corners having opening angle $\omega_j = 3 \pi/2$, so we expect $q^1 \sim r^{-1/3}$ near reentrant corners.  We correspondingly expect an a priori convergence rate of order $h^{2/3-\epsilon} \sim DOF^{-1/3+\epsilon}$ on sequences of quasi-uniform meshes, and an adaptive convergence rate of order $DOF^{-1/2}$.  These are in fact observed in the left plot of Figure \ref{fig1}.  In the right plot of Figure 1 we show the increase in $\|q_\ell^1\|_{L_\infty(\Omega)}$ as the mesh as refined, which provides confirmation that $q^1$ is singular as expected.  In Figure \ref{fig2} we display an adaptively-generated mesh showing refinement at the reentrant corners along with a representative harmonic form $q_\ell^1$ on a coarse mesh.  

 \setlength{\unitlength}{.75cm}
\begin{figure}[h]
\centering
\includegraphics[scale=.36]{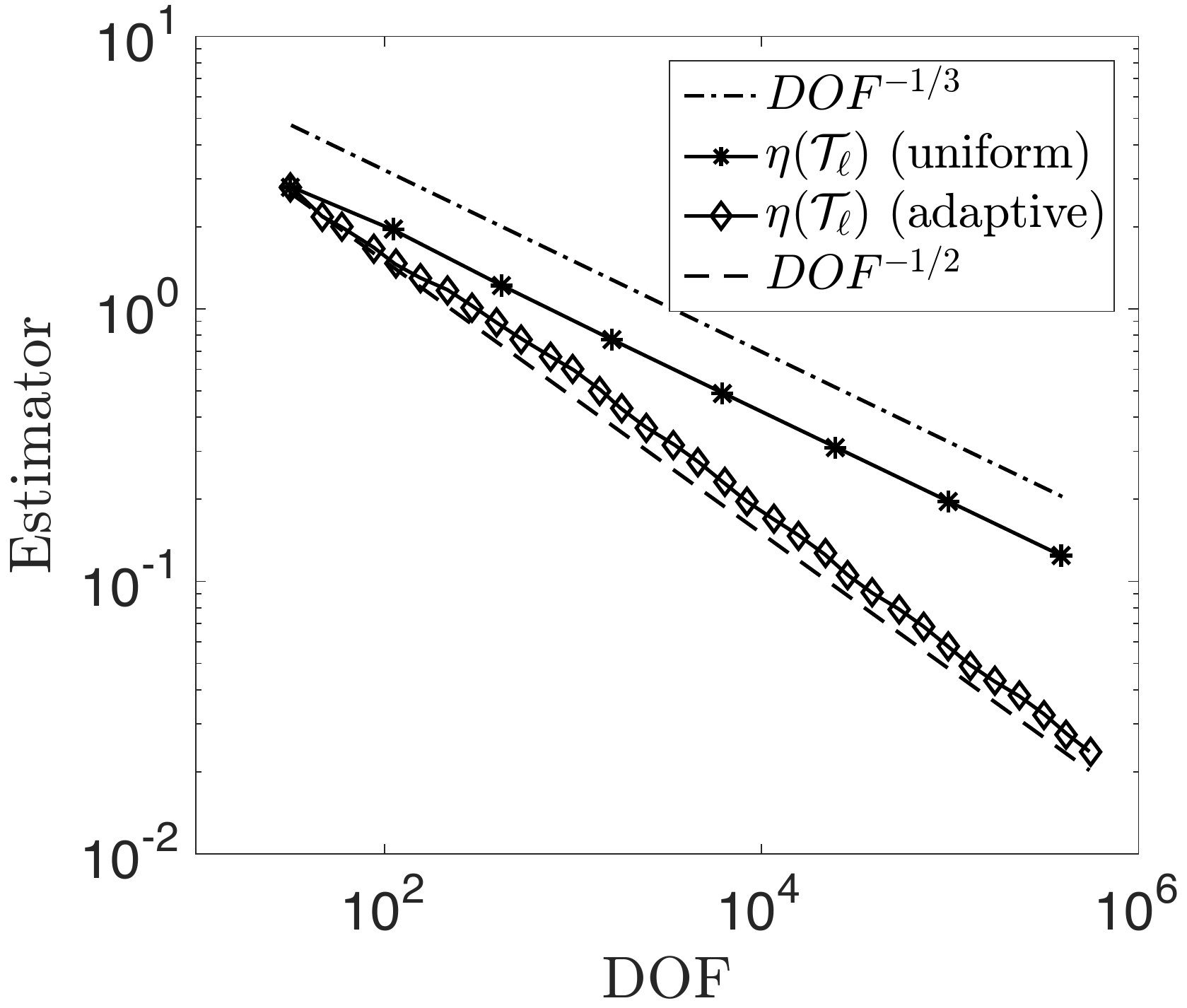}
\includegraphics[scale=.36]{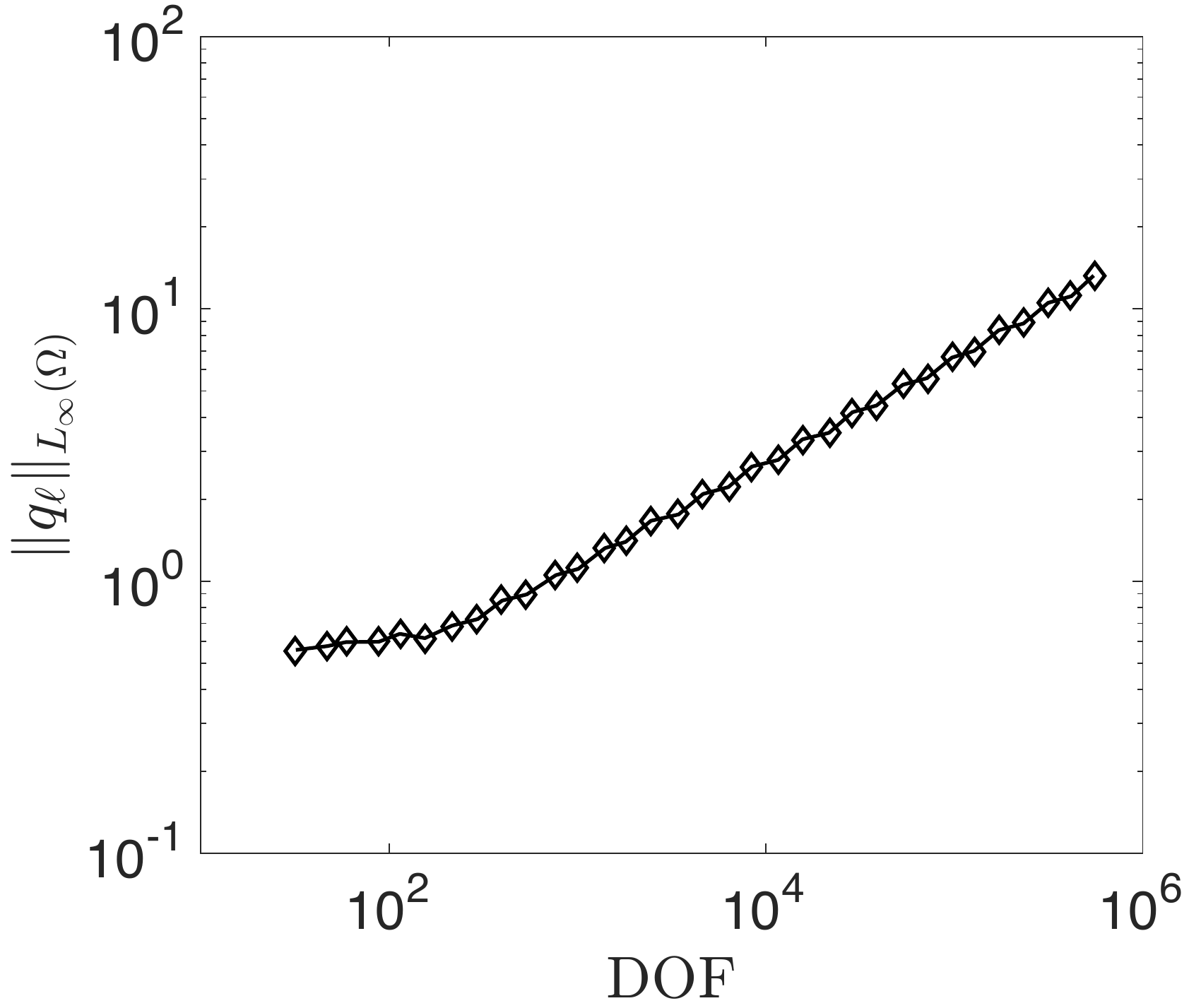}
\caption{Error decrease under adaptive and uniform refinement (left); increase in $\|q_\ell^1\|_{L_\infty(\Omega)}$ under adaptive refinement (right).}
\label{fig1}
\end{figure}

 \setlength{\unitlength}{.75cm}
\begin{figure}[h]
\centering
\includegraphics[scale=.5]{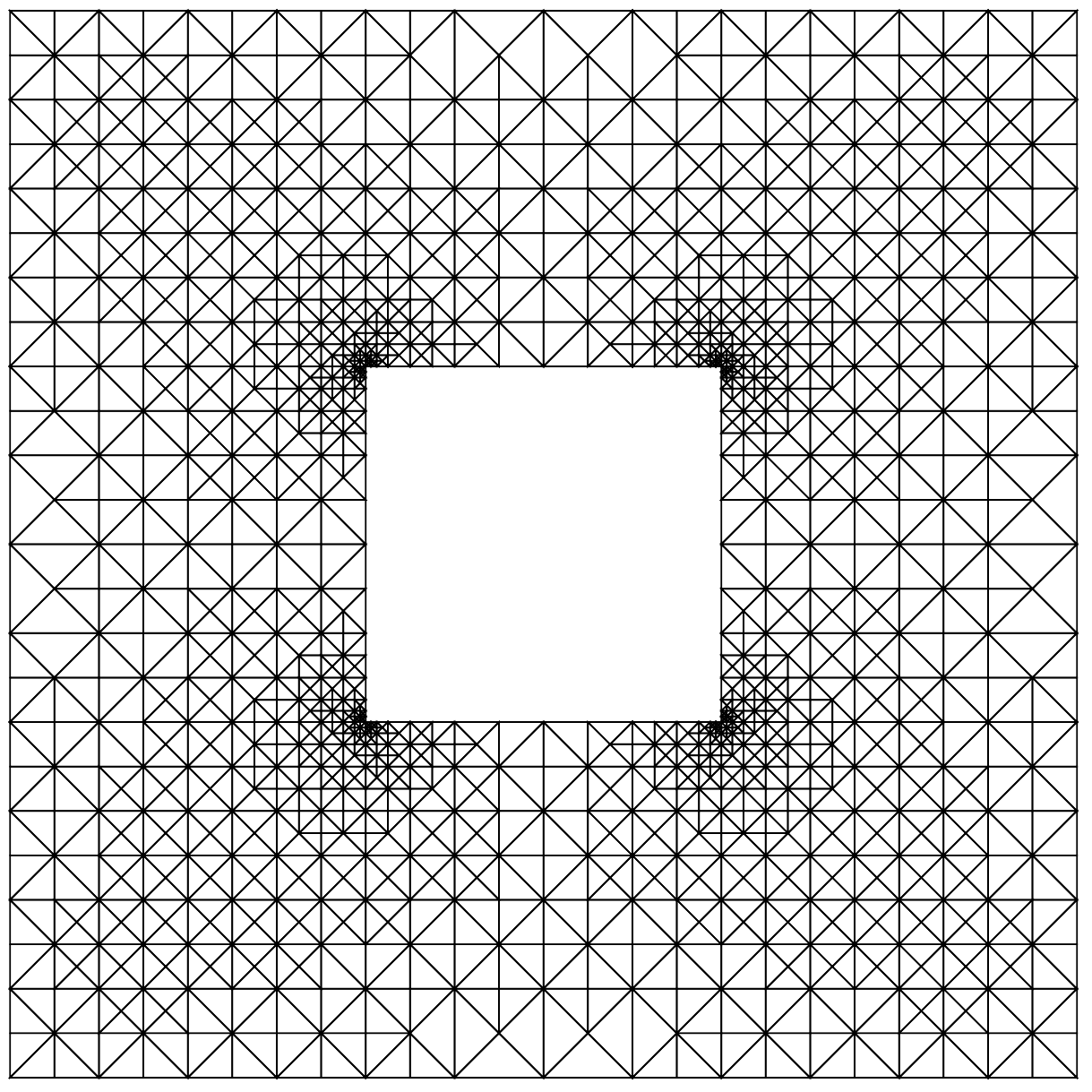}
\includegraphics[scale=.5]{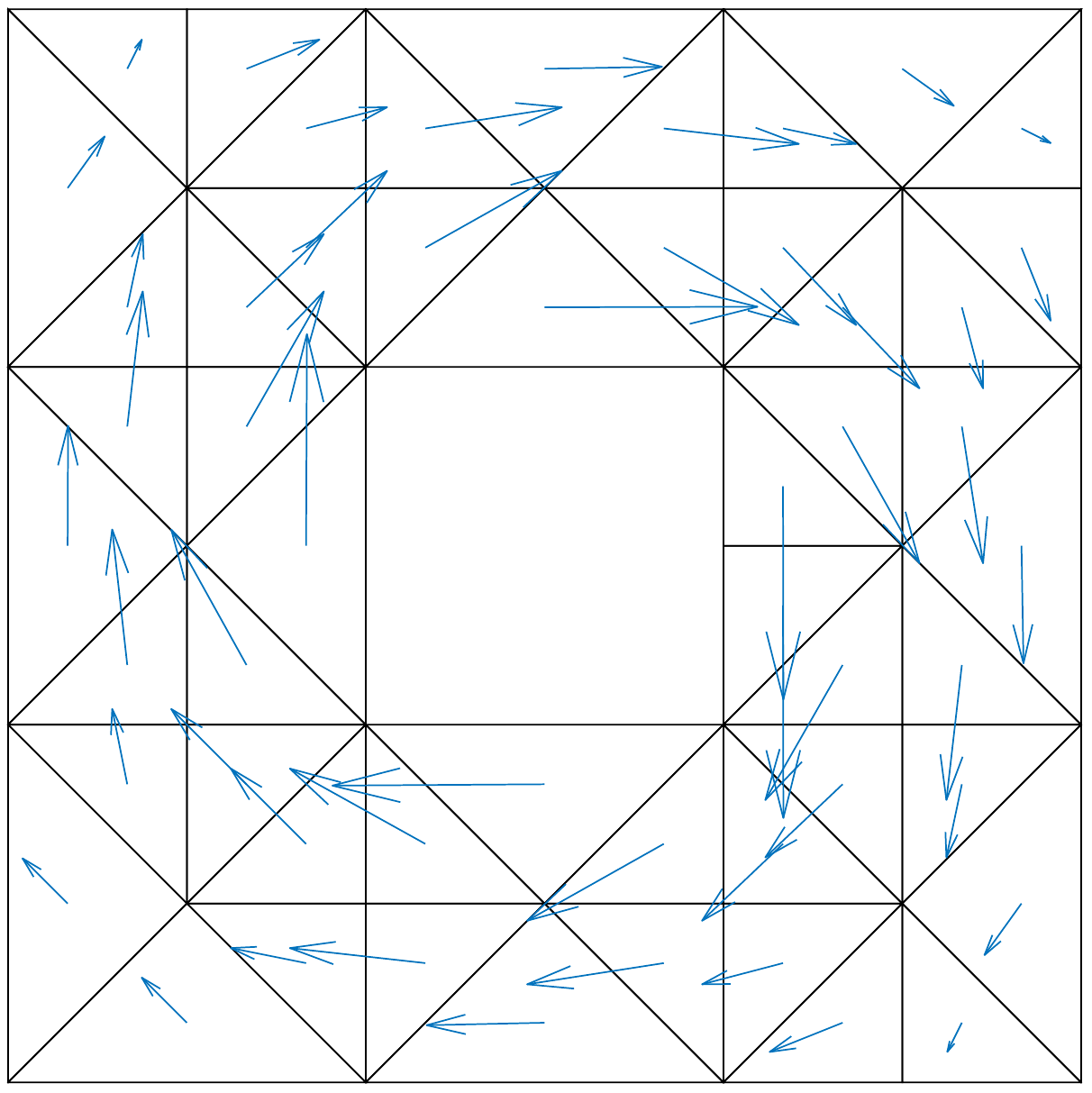}
\caption{Adaptive mesh with 2535 degrees of freedom (left); discrete harmonic form on a coarse mesh (right).}
\label{fig2}
\end{figure}

\subsection{Experiment 2: $\beta=3$.}  

In this experiment we investigate the case $\beta>1$.  In Figure \ref{fig3}, an adaptively computed basis is displayed on a relatively coarse mesh $\T_\ell$ on a domain with $\beta=3$, while Figure \ref{fig4} displays the computed discrete harmonic basis on a finer mesh $\T_{\tilde{\ell}}$ ($\tilde{\ell} > \ell$).  Comparing the two bases provides an illustration of the alignment problem discussed in the introduction.  There is no correspondence between $q_\ell^i$ and $q_{\tilde{\ell}}^i$.  Comparing two such forms in an estimator reduction inequality is not meaningful, as in the case of elliptic eigenvalue problems.  Our second comment concerns the discussion in Section \ref{sec7-3}.  There we noted that forms in a given cohomology class may not be singular at all reentrant corners.  This is illustrated by for example $q_\ell^1$ in the upper left of Figure \ref{fig3}.  The support of $q_\ell^1$ is mainly localized to the upper right quarter of $\Omega$.  Here this localization occurred somewhat randomly, but can also be manufactured by design (cf. \cite{HiKaWaWa11}).

 \setlength{\unitlength}{.75cm}
\begin{figure}[h]
\centering
\includegraphics[scale=.5]{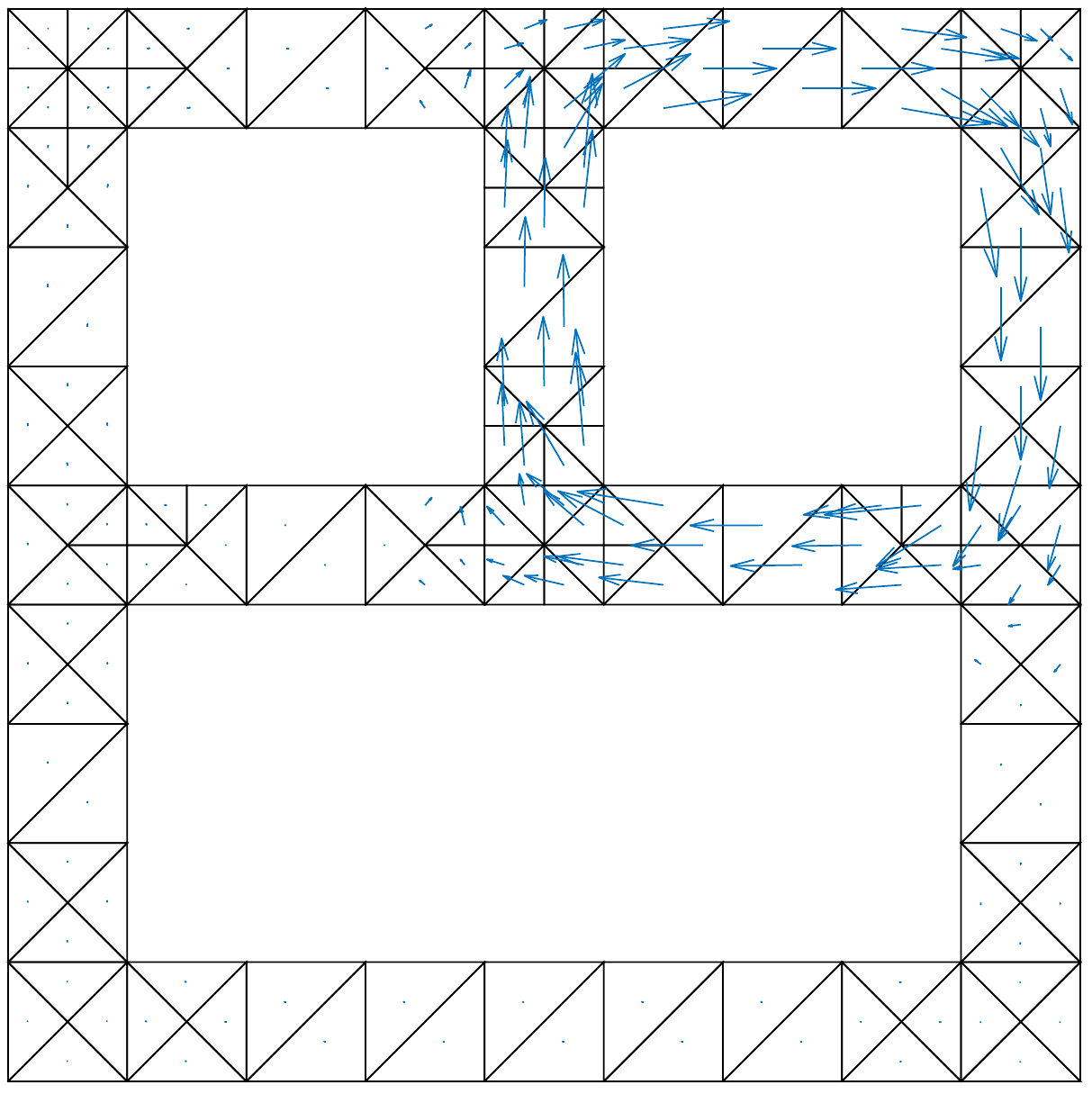}
\includegraphics[scale=.5]{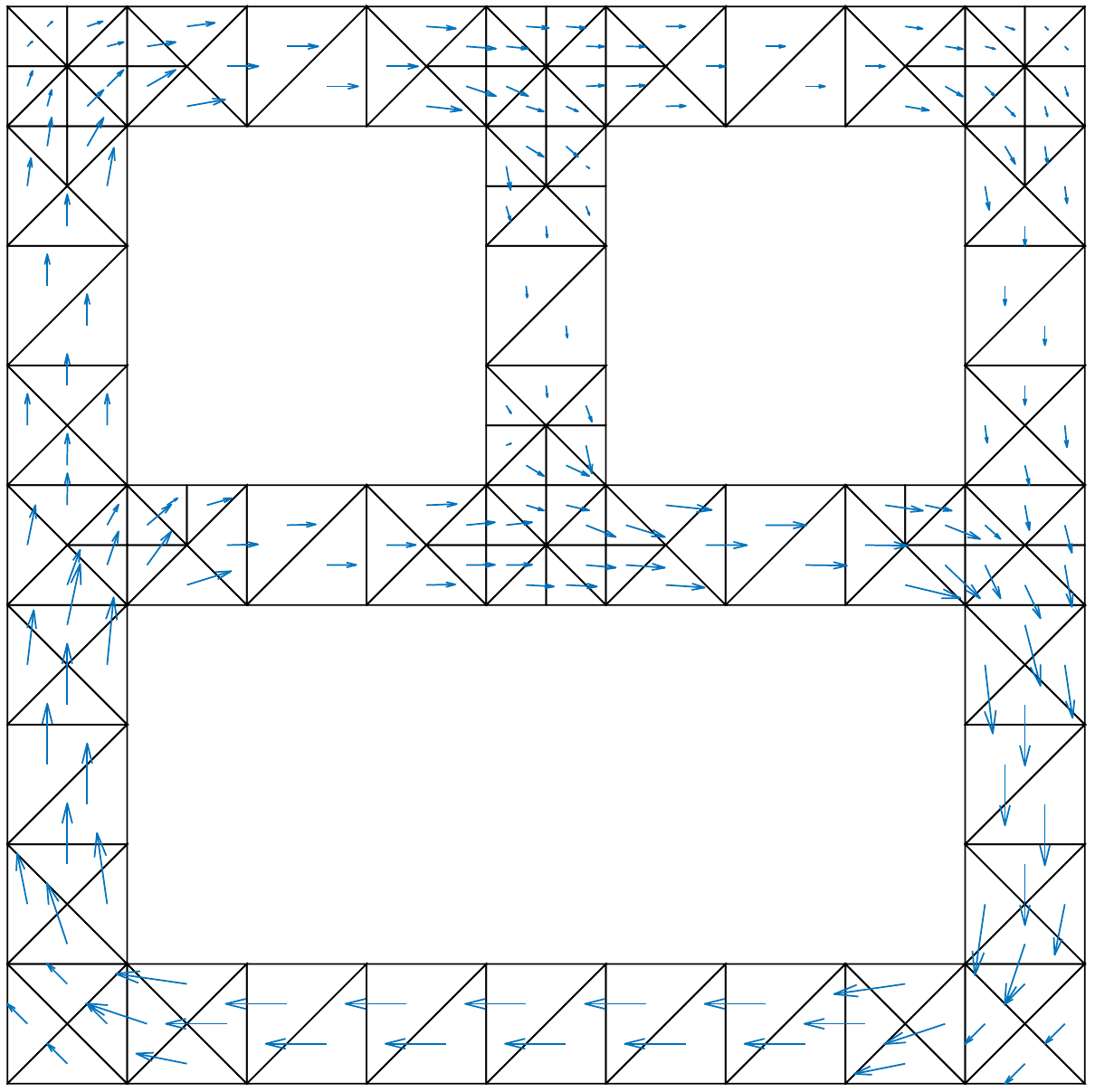}
\includegraphics[scale=.5]{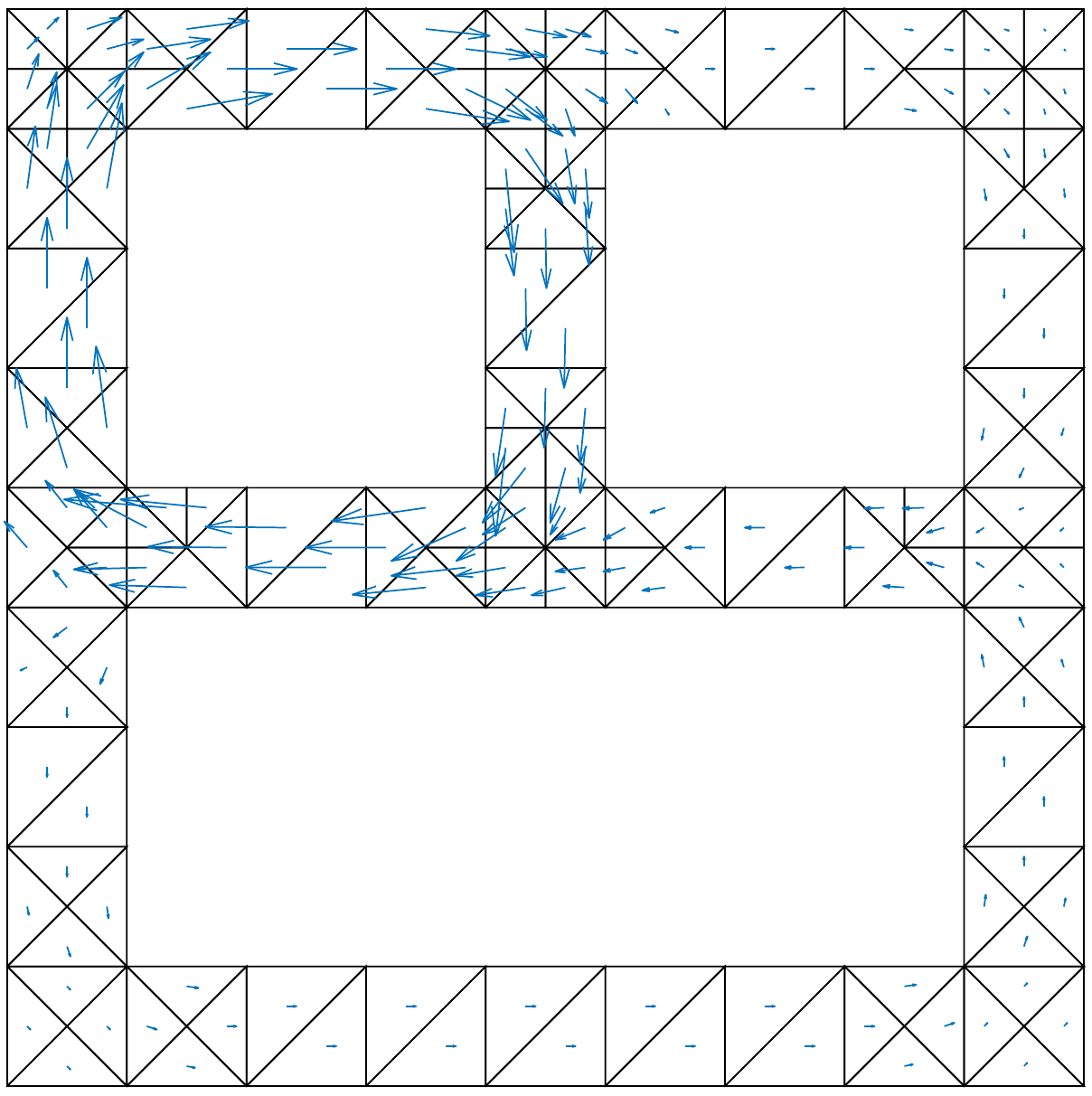}
\caption{Discrete harmonic basis elements $q_\ell^1$ (upper left), $q_\ell^2$ (upper right), and $q_\ell^3$ (lower) computed on a coarse mesh.}
\label{fig3}
\end{figure}

 \setlength{\unitlength}{.75cm}
\begin{figure}[h]
\centering
\includegraphics[scale=.5]{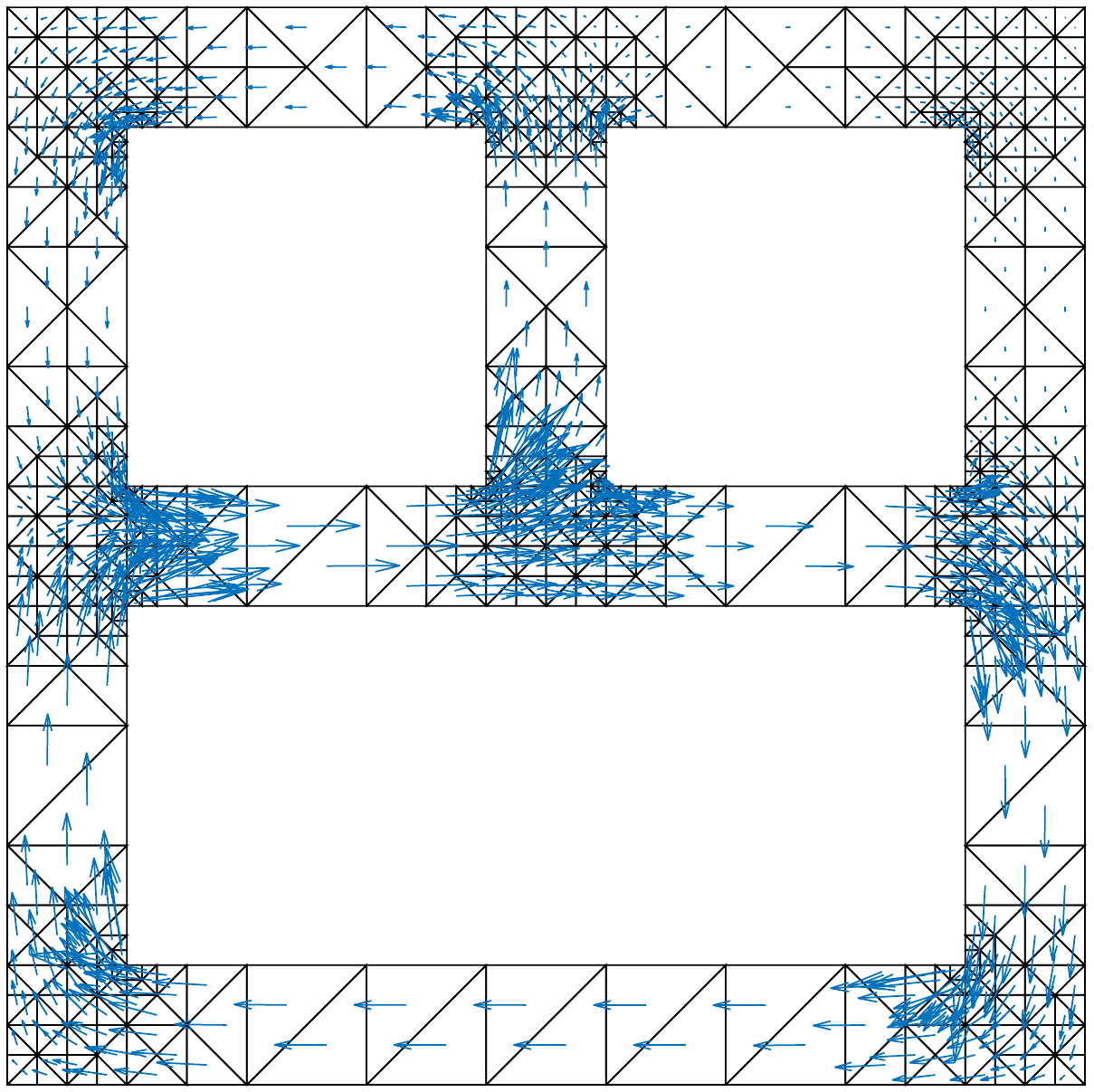}
\includegraphics[scale=.5]{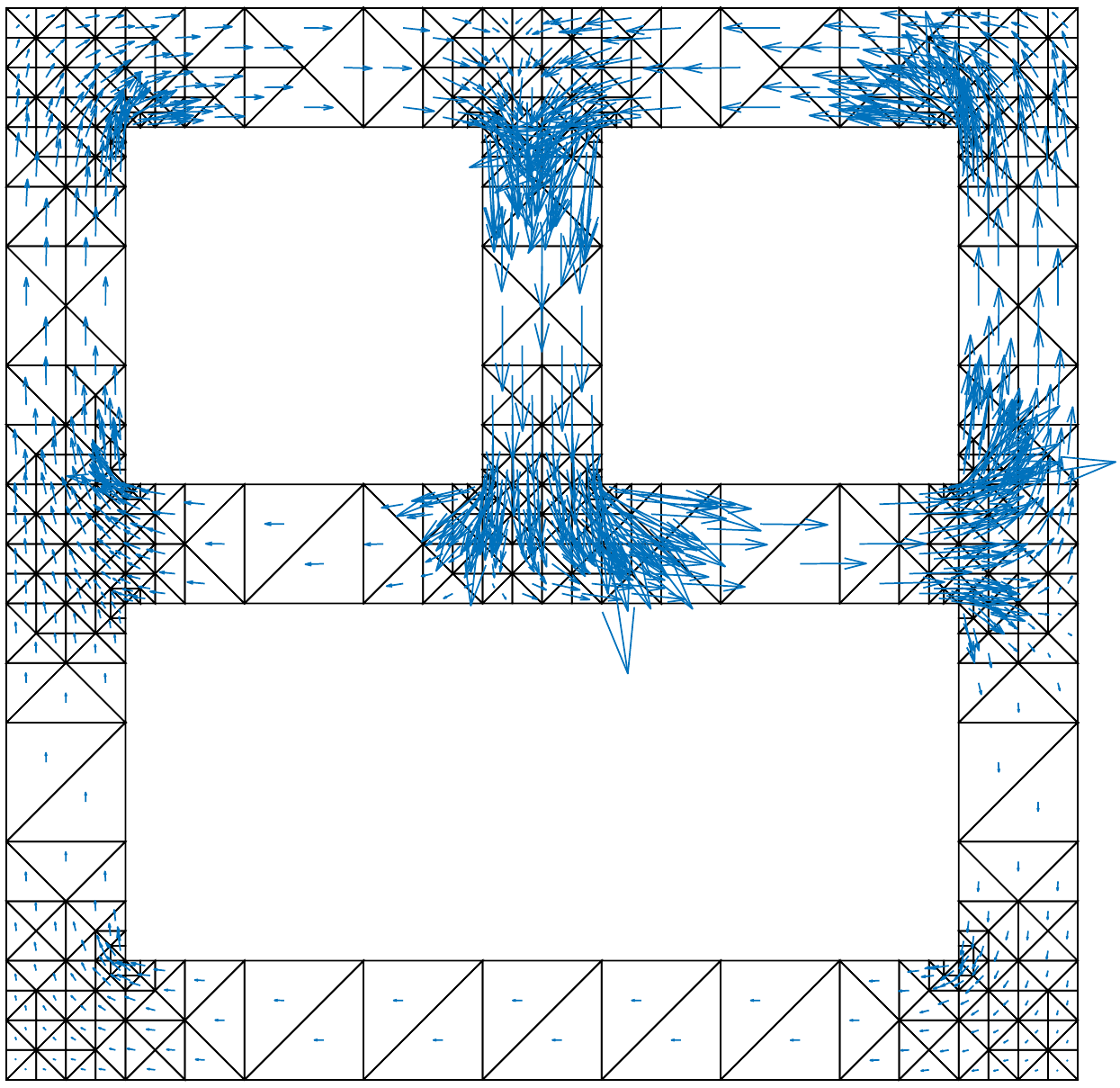}
\includegraphics[scale=.5]{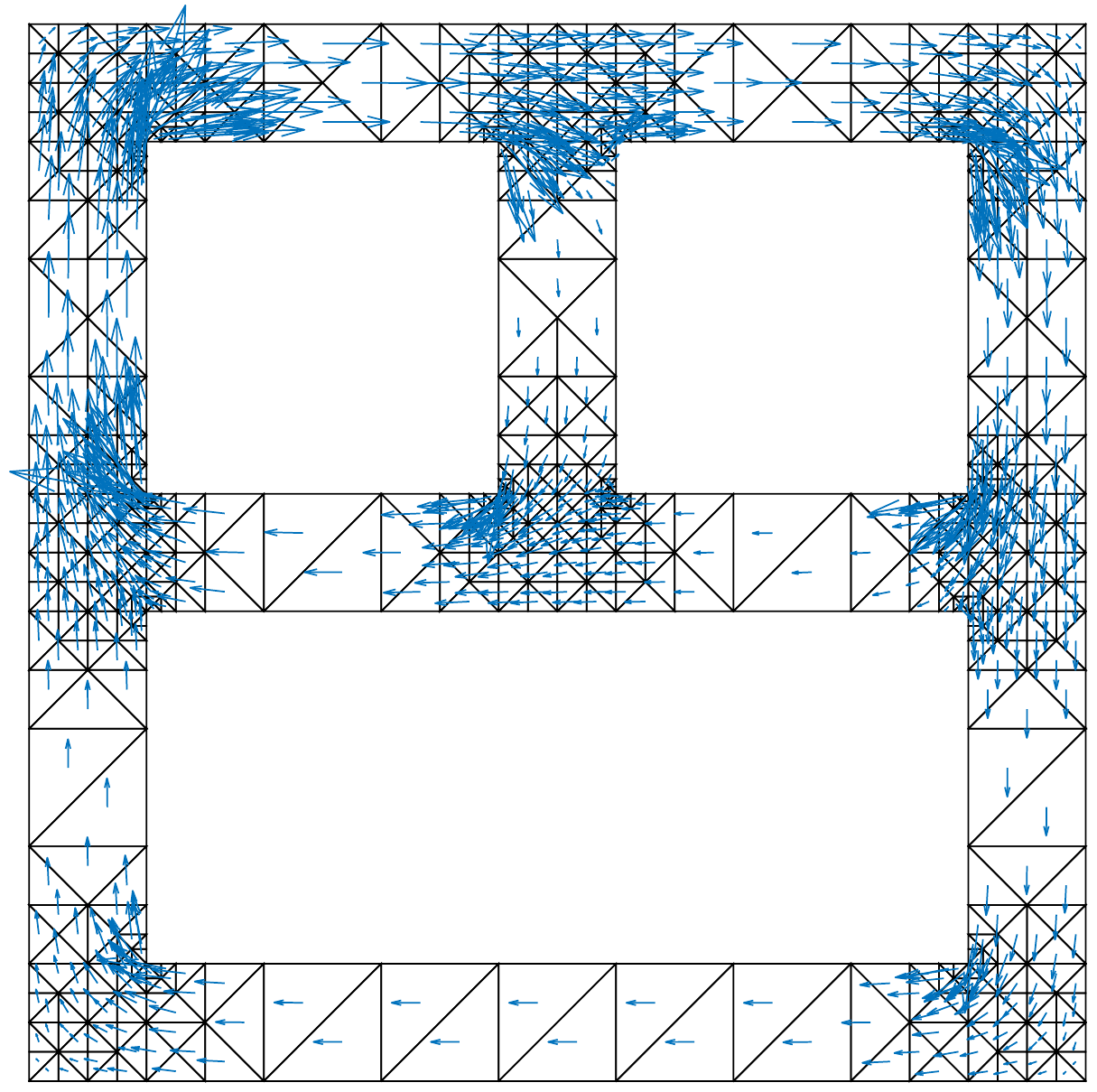}
\caption{Discrete harmonic basis elements $q_{\tilde{\ell}}^1$ (upper left), $q_{\tilde{\ell}}^2$ (upper right), and $q_{\tilde{\ell}}^3$ (lower) computed on a finer mesh.}
\label{fig4}
\end{figure}




\begin{thebibliography}{10}

\bibitem{RV10}
{\sc A.~Alonso~Rodr{\'{\i}}guez and A.~Valli}, {\em Eddy current approximation
  of {M}axwell equations}, vol.~4 of MS\&A. Modeling, Simulation and
  Applications, Springer-Verlag Italia, Milan, 2010.
\newblock Theory, algorithms and applications.

\bibitem{AFW06}
{\sc D.~N. Arnold, R.~S. Falk, and R.~Winther}, {\em Finite element exterior
  calculus, homological techniques, and applications}, Acta Numer., 15 (2006),
  pp.~1--155.

\bibitem{AFW10}
\leavevmode\vrule height 2pt depth -1.6pt width 23pt, {\em Finite element
  exterior calculus: from {H}odge theory to numerical stability}, Bull. Amer.
  Math. Soc. (N.S.), 47 (2010), pp.~281--354.

\bibitem{BFG12}
{\sc R.~Benedetti, R.~Frigerio, and R.~Ghiloni}, {\em The topology of
  {H}elmholtz domains}, Expo. Math., 30 (2012), pp.~319--375.

\bibitem{BS15}
{\sc P.~Bettini and R.~Specogna}, {\em A boundary integral method for computing
  eddy currents in thin conductors of arbitrary topology}, IEEE Transactions on
  Magnetics, 51 (2015), pp.~1--4.

\bibitem{BD15}
{\sc A.~Bonito and A.~Demlow}, {\em Convergence and optimality of higher-order
  adaptive finite element methods for eigenvalue clusters}, SIAM J. Numer.
  Anal.,  (To appear).

\bibitem{CKNS08}
{\sc J.~Cascon, C.~Kreuzer, R.~H. Nochetto, and K.~G. Siebert}, {\em
  Quasi-optimal convergence rate for an adaptive finite element method}, SIAM
  J. Numer. Anal., 46 (2008), pp.~2524--2550.

\bibitem{Ch09PP}
{\sc L.~Chen}, {\em {iFEM}: An innovative finite element method package in
  {M}atlab}, tech. rep., University of California-Irvine, 2009.

\bibitem{CW08}
{\sc S.~H. Christiansen and R.~Winther}, {\em Smoothed projections in finite
  element exterior calculus}, Math. Comp., 77 (2008), pp.~813--829.

\bibitem{CD00}
{\sc M.~Costabel and M.~Dauge}, {\em Singularities of electromagnetic fields in
  polyhedral domains}, Arch. Ration. Mech. Anal., 151 (2000), pp.~221--276.

\bibitem{DHZ15}
{\sc L.~Dai, Xiaoying;~He and A.~Zhou}, {\em Convergence and quasi-optimal
  complexity of adaptive finite element computations for multiple eigenvalues},
  IMA J. Numer. Anal.,  (2015).

\bibitem{DXZ08}
{\sc X.~Dai, J.~Xu, and A.~Zhou}, {\em Convergence and optimal complexity of
  adaptive finite element eigenvalue computations}, Numer. Math., 110 (2008),
  pp.~313--355.

\bibitem{dauge_web}
{\sc M.~Dauge}, {\em Regularity and singularities in polyhedral domains}.
\newblock
  https://perso.univ-rennes1.fr/monique.dauge/publis/Talk\_Karlsruhe08.pdf,
  April 2008.

\bibitem{DH14}
{\sc A.~Demlow and A.~N. Hirani}, {\em A posteriori error estimates for finite
  element exterior calculus: the de {R}ham complex}, Found. Comput. Math., 14
  (2014), pp.~1337--1371.

\bibitem{DST09}
{\sc P.~D{\l}otko, R.~Specogna, and F.~Trevisan}, {\em Automatic generation of
  cuts on large-sized meshes for the {$T$}-{$\Omega$} geometric eddy-current
  formulation}, Comput. Methods Appl. Mech. Engrg., 198 (2009), pp.~3765--3781.

\bibitem{FW14}
{\sc R.~S. Falk and R.~Winther}, {\em Local bounded cochain projections}, Math.
  Comp., 83 (2014), pp.~2631--2656.

\bibitem{FSDH07}
{\sc M.~Fisher, P.~Schr\"oder, M.~Desbrun, and H.~Hoppe}, {\em Design of
  tangent vector fields}, ACM Transactions on Graphics, 26 (2007),
  pp.~56--1--56--9.

\bibitem{Gal15}
{\sc D.~Gallistl}, {\em An optimal adaptive {FEM} for eigenvalue clusters},
  Numer. Math., 130 (2015), pp.~467--496.

\bibitem{GMZ09}
{\sc E.~M. Garau, P.~Morin, and C.~Zuppa}, {\em Convergence of adaptive finite
  element methods for eigenvalue problems}, Math. Models Methods Appl. Sci., 19
  (2009), pp.~721--747.

\bibitem{GG09}
{\sc S.~Giani and I.~G. Graham}, {\em A convergent adaptive method for elliptic
  eigenvalue problems}, SIAM J. Numer. Anal., 47 (2009), pp.~1067--1091.

\bibitem{Hip02}
{\sc R.~Hiptmair}, {\em Finite elements in computational electromagnetism},
  Acta Numer., 11 (2002), pp.~237--339.

\bibitem{HO02}
{\sc R.~Hiptmair and J.~Ostrowski}, {\em Generators of {$H_1(\Gamma_h,\Bbb Z)$}
  for triangulated surfaces: construction and classification}, SIAM J. Comput.,
  31 (2002), pp.~1405--1423.

\bibitem{HX07}
{\sc R.~Hiptmair and J.~Xu}, {\em Nodal auxiliary space preconditioning in
  {${\bf H}({\bf curl})$} and {${\bf H}({\rm div})$} spaces}, SIAM J. Numer.
  Anal., 45 (2007), pp.~2483--2509 (electronic).

\bibitem{HiKaWaWa11}
{\sc A.~N. Hirani, K.~Kalyanaraman, H.~Wang, and S.~Watts}, {\em Cohomologous
  harmonic cochains}, June 2011.
\newblock Available as e-print on arxiv.org.

\bibitem{MHS13}
{\sc M.~{Holst}, A.~{Mihalik}, and R.~{Szypowski}}, {\em {Convergence and
  Optimality of Adaptive Methods in the Finite Element Exterior Calculus
  Framework}}, ArXiv e-prints,  (2013).

\bibitem{MiMiMo08}
{\sc D.~Mitrea, M.~Mitrea, and S.~Monniaux}, {\em The {P}oisson problem for the
  exterior derivative operator with {D}irichlet boundary condition in nonsmooth
  domains}, Commun. Pure Appl. Anal., 7 (2008), pp.~1295--1333.

\bibitem{RBGV13}
{\sc A.~A. Rodr{\'\i}guez, E.~Bertolazzi, R.~Ghiloni, and A.~Valli}, {\em
  Construction of a finite element basis of the first de {R}ham cohomology
  group and numerical solution of 3{D} magnetostatic problems}, SIAM J. Numer.
  Anal., 51 (2013), pp.~2380--2402.

\bibitem{Sch01}
{\sc J.~Sch{\"o}berl}, {\em Commuting quasi-interpolation operators for mixed
  finite elements}, Tech. Rep. ISC-01-10-MATH, Institute for Scientific
  Computing, Texas A\&M University, 2001.

\bibitem{Sch08}
{\sc J.~Sch{\"o}berl}, {\em A posteriori error estimates for {M}axwell
  equations}, Math. Comp., 77 (2008), pp.~633--649.

\bibitem{SZ90}
{\sc L.~R. Scott and S.~Zhang}, {\em Finite element interpolation of nonsmooth
  functions satisfying boundary conditions}, Math. Comp., 54 (1990),
  pp.~483--493.

\bibitem{Ste07}
{\sc R.~Stevenson}, {\em Optimality of a standard adaptive finite element
  method}, Found. Comput. Math., 7 (2007), pp.~245--269.

\bibitem{Ste08}
\leavevmode\vrule height 2pt depth -1.6pt width 23pt, {\em The completion of
  locally refined simplicial partitions created by bisection}, Math. Comp., 77
  (2008), pp.~227--241 (electronic).

\bibitem{TV14}
{\sc A.~Valli and F.~Tr\"oltzsch}, {\em Optimal control of low-frequency
  electromagnetic fields in multiply connected conductors}, tech. rep., DFG
  Research Center Matheon, 2014.

\bibitem{XZCX09}
{\sc K.~Xu, H.~Zhang, D.~Cohen-Or, and Y.~Xiong}, {\em Dynamic harmonic fields
  for surface processing}, Computers {\&} Graphics, 33 (2009), pp.~391--398.

\bibitem{ZCSWX11}
{\sc L.~Zhong, L.~Chen, S.~Shu, G.~Wittum, and J.~Xu}, {\em Convergence and
  optimality of adaptive edge finite element methods for time-harmonic
  {M}axwell equations}, Math. Comp., 81 (2012), pp.~623--642.

\end{thebibliography}

\end{document}